\documentclass[oneside]{amsart}
\usepackage[utf8]{inputenc}
\usepackage{amstext}
\usepackage{amsthm}
\usepackage{amssymb}

\makeatletter
\numberwithin{equation}{section}
\numberwithin{figure}{section}
\newenvironment{lyxlist}[1]
	{\begin{list}{}
		{\settowidth{\labelwidth}{#1}
		 \setlength{\leftmargin}{\labelwidth}
		 \addtolength{\leftmargin}{\labelsep}
		 }}
	{\end{list}}

\usepackage{enumitem}
\setlist[enumerate,1]{label={\(\mathrm{(\arabic*)}\)}, align=left, left=0pt}

\numberwithin{equation}{section}
\theoremstyle{plain}

\theoremstyle{definition}

\DeclareMathOperator{\Ul}{Ul}

\DeclareMathOperator{\Clop}{Clop}
\DeclareMathOperator{\Int}{Int}

\usepackage{xcolor}
\usepackage[
  colorlinks=true,
  linkcolor=blue,    
  citecolor=blue,   
  urlcolor=red       
]{hyperref}

\makeatother

\theoremstyle{plain}
\newtheorem{thm}{\protect\theoremname}
\theoremstyle{definition}
\newtheorem{defn}[thm]{\protect\definitionname}
\theoremstyle{plain}
\newtheorem{lem}[thm]{\protect\lemmaname}
\theoremstyle{remark}
\newtheorem{rem}[thm]{\protect\remarkname}
\theoremstyle{plain}
\newtheorem{prop}[thm]{\protect\propositionname}
\newtheorem{cor}[thm]{\protect\corollaryname}
\providecommand{\corollaryname}{Corollary}
\providecommand{\definitionname}{Definition}
\providecommand{\lemmaname}{Lemma}
\providecommand{\propositionname}{Proposition}
\providecommand{\remarkname}{Remark}
\providecommand{\theoremname}{Theorem}

\begin{document}

\title[A bitopological duality for some]{A bitopological duality for some  subordination Boolean algebras}



\author[S. A. Celani]{Sergio A. Celani}
\address{Departamento de Matemática\\
Facultad de Ciencias Exactas\\
Universidad Nacional del Centro. CONICET\\
Pinto 390\\Tandil. Argentina\\
ORCID:https://orcid.org/0000-0003-2542-4128}

\email{scelani@exa.unicen.edu.ar}




\begin{abstract}
 S4-subordination algebras are a generalization of the closure algebras.
In this paper, we give a topological representation for S4-subordination
algebras by means of bitopological spaces $\left<X,\tau,\tau_{S}\right>$,
where $\left<X,\tau\right>$ is a Stone space and $\tau_{S}$ is a
topology that enables the characterization of the subordination relation.
We apply this bitopological representation to give a characterization
of S5-subordination algebras and lattice subordinations. We also show
that there exists a bijective correspondence between congruence compatible
with the subordination and certain closed subsets of the Stone space
$\left<X,\tau\right>$ that are also saturated sets of the space $\left<X,\tau_{S}\right>$.
Additionally, we explore two types of morphisms between S4-subordination
algebras: one based on Boolean homomorphisms and another based on
meet-homomorphisms. Finally, we provide a topological representation
for each type of morphism.
\end{abstract}

\maketitle
\maketitle

\section{Introduction}

The classes of subordination Boolean algebras $\mathsf{Sub}$ \cite{Bezhanesville-Nick and venema},
precontact algebras \cite{D=0000FCntsch-Vakarelov2007} $\mathsf{PreC}$,
and quasi-modal algebras $\mathsf{QMA}$ \cite{Celani QUasi modal},
provide three equivalent presentations of the same first-order structure.
Each of these frameworks can be regarded as a generalization of compingent
algebras \cite{de Vries tesis}, contact algebras \cite{Dimo-Vakarelov-Contact algebras and region-based theory of space,G. Dimov and D. Vakarelov-Topological representation of precontact algebras,Dimov-Vakarelov-Topological representation of Stone I,Vakarelov-Region-based theory of space- Algebras of regions-representation theory},
and modal algebras \cite{Sambin}. 

It is well known that the category of modal algebras is dually equivalent
to the category of Stone spaces endowed with a binary relation satisfying
suitable conditions \cite{Sambin}. This duality can also be obtained
via the correspondence between relations on Stone spaces and 0-preserving
join-homomorphisms between Boolean algebras of clopen sets, as established
by Halmos \cite{Halmos1962} and by Wright \cite{Wright}. Such relational
dualities were later extended to subordination algebras in \cite{Celani QUasi modal}
, \cite{Dimov-Vakarelov-Topological representation of Stone I}, and
\cite{G. Dimov and D. Vakarelov-Topological representation of precontact algebras}.

Among subordination algebras, the class of S4-subordination algebras
plays a central role, as it generalizes closure algebras (also known
as interior algebras or topological Boolean algebras). This class
of structures was first introduced in \cite[Theorem 26]{Celani QUasi modal}
and \cite[Definition 22]{Celani-SubIrreQuasi} under the name of quasi-topological
algebras. Its relational duality has been studied from several perspectives.
In particular, \cite[Theorem 26]{Celani QUasi modal} establishes
dual equivalence between S4-subordination algebras and Stone spaces
endowed with a binary reflexive and transitive relation $R$ satisfying
suitable conditions; see also \cite[Theorem 4.11(2)]{Bezhanesville-Nick and venema}
and \cite{Dimov-Vakarelov-Topological representation of Stone I}.

In \cite{Bezhanishvili-Mines-Morandi 2007} Bezhanishvili, Mines,
and Morandi combine McKinsey-Tarski representation of closure algebras
with Stone duality for Boolean algebras to obtaina duality between
closure algebras and spaces equipped with both the Stone topology
and the McKinsey--Tarski topology. Since S4-subordination algebras
extend closure algebras, it is natural to seek a corresponding extension
of this bitopological duality.

The aim of the present paper is to develop a genuinely bitopological
duality for S4-subordination algebras. The dual objects are bitopological
spaces $\langle X,\tau,\tau_{S}\rangle$, where $\tau$ is the Stone
topology of the underlying Boolean algebra and $\tau_{S}$ is an auxiliary
topology whose basic open sets are the interiors $\mathrm{int}_{S}(U)$,
for $U\in\mathrm{Clop}(X)$, and which encodes the subordination relation.
More precisely, given a S4-subordination algebra $\langle A,\prec\rangle$,
the basic open sets of $\tau_{S}$ are the sets $S=\left\{ x\in\Ul(A):x\cap\downarrow a\neq\emptyset\right\} $,
for $a\in A$, where $\downarrow a=\{b\in A:b\prec a\}$. With this
construction, the subordination relation is recovered by the condition
$a\prec b$ if and only if $\varphi(a)\subseteq\mathrm{int}_{S}(\varphi(b))$,
where $\varphi$ denotes the Stone map.In the relational language
of \cite{Bezhanesville-Nick and venema,Celani QUasi modal}, this
corresponds to the condition $R[\varphi(a)]\subseteq\varphi(b)$.
Thus, the binary relation $R$ appearing in the relational dualities
of \cite{Bezhanesville-Nick and venema,Celani QUasi modal} is precisely
the specialization order of the topology $\tau_{S}$.

In the literature, there are well-established algebraic, topological,
and logical motivations for considering different notions of morphisms
between algebraic structures and their corresponding dual spaces.
A classical example is the mentioned Halmos-Wright duality  (see
\cite{Halmos1962} and \cite{Wright}). Since a modal operator in
a Boolean algebra is a meet-homomorphism on the underlying algebra,
the topological representation of modal algebras can be obtained from
the results of Halmos and Wright (see \cite{Sambin}). Further examples
arise in the study of de Vries algebras, referred to as Boolean $\delta$-algebras
in \cite{Fedorchuk1973}. De Vries \cite{de Vries tesis} showed that
the category of compact Hausdorff spaces and continuous maps is dually
equivalent to the category of complete compingent Boolean algebras
equipped with suitable morphisms, known as de Vries morphisms. However,
the composition of de Vries morphisms does not coincide with their
set-theoretic composition. In \cite{Fedorchuk1973}, Fedorchuk observed
that de Vries morphisms which are complete Boolean homomorphisms admit
a particularly simple description and, in this restricted setting,
the composition of two such morphisms agrees with their set-theoretic
composition. He further considered the category of complete compingent
Boolean algebras with complete de Vries morphisms and proved that
it is dually equivalent to the category of compact Hausdorff spaces
and quasi-open maps. In this paper, we follow a similar approach by
considering different classes of homomorphisms between S4-subordination
algebras and describing their corresponding dual morphisms between
bitopological spaces.

The paper is organized as follows. In Section \ref{sec:Preliminaries},
we recall the necessary background on subordination relation, quasi-modal
operator, and the basic concepts on the Stone duality. In Section
\ref{sec:Topological-representation} we introduce S4-subordination
spaces and establish the main representation theorem. We will also
extend the representation to the class $\mathsf{SubS5}$ of S5-subordination
algebras (also known as quasi-monadic algebras in \cite{Celani-SubIrreQuasi})
and de Vries algebras. In Section \ref{sec:Lattice-subordinations}
we will characterize the class of lattice subordination \cite{Bezhanishvili2013}.
In Section \ref{sec:Subordination-congruences} we will apply the
topological representation to characterize the congruences compatible
with the subordination. The notion of congruence compatible with a
subordination was introduced in \cite{Celani QUasi modal} and \cite{Celani-SubIrreQuasi}
in terms of quasi-modal operators as a generalization of the notion
of congruences of modal algebras.

As an illustration of the flexibility of the bitopological approach,
we consider two natural choices of morphisms between S4-subordination
algebras. In Section \ref{sec:Categorical-dualities-I} we will study
morphisms between $S4$-subordination algebras based in Boolean homomorphisms,
obtaining a duality between the categories $\mathbf{SubS4}_{h}$ and
$\mathbf{BiSpS4}_{s}$, and between the categories $\mathbf{SubS4}_{st}$
 and $\mathbf{BiSpS4}_{st}$. In Section \ref{sec:Subordination-meet-hemimorphisms}
we study morphisms based on meet-homomorphisms, which is the natural
choice in the context of de Vries duality \cite{de Vries tesis}:
under de Vries duality, the dual morphisms of continuous maps between
compact Hausdorff spaces are meet-homomorphisms satisfying additional
conditions. 

\section{Preliminaries}\label{sec:Preliminaries}

In this section, we recall basic notions of Boolean algebras, with
particular emphasis on their topological representation \cite{Balbes}. 

Let $A$ be a Boolean algebra. A \textit{{filter}} is a subset $F$
of $A$ such that $1\in F$, if $a\leq b$ and $a\in F$, then $b\in F$,
and if $a,b\in F$, then $a\wedge b\in F$. An \textit{{ideal}}
is a subset $I$ of $A$ such that $0\in I$, if $a\leq b$ and $b\in I$,
then $a\in I$, and if $a,b\in I$, then $a\vee b\in F$. Denote by
${\rm {Fi}}(A)$ and ${\rm {Id}}(A)$ the set of all filters and the
set of all ideals of $A$, respectively. The \textit{{ideal (filter)
generated by a subset $X\subseteq A$}} is denoted by ${\rm {Idg}}(X)$
(${\rm {Fig}}(X)$). In particular, if $X=\{x\}$, then we write ${\rm {Idg}}(\{x\})=(x]=\{a\in A\colon a\leq x\}$
(${\rm {Fig}}(\{x\})=[x)=\{a\in A\colon x\leq a\}$). A proper filter
$P$ of $A$ is \textit{{prime}} if for each $a,b\in A$, $a\vee b\in P$
implies $a\in P$ or $b\in P$. It is known that a filter $P$ of
a Boolean algebra $A$ is prime iff is an ultrafilter. 

Let $\Ul(A)$ the set of all ultrafilters of a Boolean algebra $A$.
Let $\varphi:A\rightarrow\mathcal{P}(\Ul(A))$ be the map defined
by $\varphi(a)=\left\{ x\in\Ul(A):a\in X\right\} .$ It is known that
the family $\varphi\left[A\right]=\left\{ \varphi(a):a\in A\right\} $
is a base for a topology $\tau_{A}$ on $\Ul(A)$. The topological
space $\left<\Ul(A),\tau_{A}\right>$ is a Stone space, i.e., it is
compact, Hausdorff and totally disconnected. If $\left<X,\tau\right>$
is a Stone space, then the family $\Clop(X)$ of all closed and open
sets of $\left<X,\tau\right>$ is a base for $X$ and it is a Boolean
algebra under set-theoretical complement and intersection. Also, the
application $\varepsilon:X\rightarrow\Ul(\Clop(X))$ given by $\varepsilon(x)=\left\{ U\in\Clop(X):x\in U\right\} $
is a bijective and continuous function. With each Boolean algebra
$A$ we can associate a Stone space whose points are the elements
of $\Ul(A)$ with the topology determined by the base $\varphi\left[A\right]=\left\{ \varphi(a):a\in A\right\} $.
By the above considerations we have that, if $X$ is a Stone space,
then $X\cong\Ul(\Clop(X))$, and if $A$ is a Boolean algebra, then
$A\cong\Clop(\Ul(A))$ \cite{Balbes}.

It is known that if $A$ is a Boolean algebra and $\text{\ensuremath{\left<X,\tau\right>}}$
is the Stone space $A$, then there exists a duality between ideals
(filters) of $A$ and open (closed) sets. More precisely, for $I\in\mathrm{Id}(A)$,
\[
\beta(I)=\left\{ x\in X:I\cap x\neq\emptyset\right\} 
\]
 is an open subset of $\left<X,\tau\right>$, and for each $F\in\mathrm{Fi}(A)$,
\[
\alpha(F)=\left\{ x\in X:F\subseteq x\right\} 
\]
 is an closed subset of $\left<X,\tau\right>$. We note that $\beta(I)=\bigcup\left\{ \varphi(a):a\in I\right\} $
and $\alpha(F)=\bigcap\left\{ \varphi(a):a\in F\right\} $.

\begin{defn}
\cite{Bezhanesville-Nick and venema} A \emph{subordination on a Boolean
algebra} $A$ is a binary relation\emph{ $\prec\subseteq A\times A$
}satisfying the following conditions:
\begin{enumerate}
\item[(S1)] $0\prec0$ and $1\prec1$, 
\item[(S2)] $a\prec b,c$ implies $a\prec b\wedge c$, 
\item[(S3)] $a,b\prec c$ implies $a\vee b\prec c$, 
\item[(S4)] $a\leq b\prec c\leq d$ implies $a\prec d$.
\end{enumerate}
\end{defn}

We denote by $\mathsf{Sub}$ the class of subordination algebras. 

Let $A\in\mathsf{Sub}$, and $a\in A$ and $C\subseteq A$. We define
the following subsets of $A$:
\begin{enumerate}
\item ${\uparrow}a=\left\{ b\in A:a\prec b\right\} $ and ${\downarrow}a=\left\{ b\in A:b\prec a\right\} $.
\item ${\uparrow}C=\bigcup\left\{ \uparrow a:a\in C\right\} $ and ${\downarrow}C=\bigcup\left\{ {\downarrow}a:a\in C\right\} $. 
\end{enumerate}
\begin{defn}
\cite{Celani QUasi modal} Let $A$ be a Boolean algebra. A \emph{quasi-modal
operator} on $A$ is a map $\Delta:A\rightarrow\mathrm{Id}(A)$ such
that it satisfies the following conditions for all\emph{ }$a,b\in A:$
\begin{enumerate}
\item[Q1] $\Delta(a\wedge b)=\Delta a\cap\Delta b$,
\item[Q2] $\Delta1=A$. 
\end{enumerate}
\end{defn}

A pair $\left<A,\Delta\right>$, where $\Delta:A\rightarrow\mathrm{Id}(A)$
is a quasi-modal operator and $A$ is a Boolean algebra is called
a \emph{quasi-modal algebra.} A quasi-modal operator $\Delta$ is
monotonic, because if $a\leq b$, then $a=a\wedge b$, and so $\Delta a=\Delta(a\wedge b)=\Delta a\cap\Delta b$,
i.e., $\Delta a\subseteq\Delta b$$.$ 

In Boolean algebras, the notions of subordination relations and quasi-modal
operator are equivalent \cite{Bezhanesville-Nick and venema}\cite{Celani- A survey a contact}.
If $\prec$ is a subordination in a Boolean algebra $A$, then by
defining $\Delta_{\prec}a={\downarrow}a$, for each $a\in A$, we
obtain that $\Delta_{\prec}$ is a quasi-modal operator on $A$. Conversely,
if $\Delta$ is a quasi-modal operator on $A$, then the relation
$\prec_{\Delta}$ given by $a\prec_{\Delta}b$ iff $a\in\Delta b$,
is a subordination relation on $A$. In this paper we will mainly
work with subordinations instead of quasi-modal operators.
\begin{lem}
\label{lem:filter and ideal}\cite{Bezhanesville-Nick and venema,Celani QUasi modal}
Let $A\in\mathsf{Sub}$. Then
\begin{enumerate}
\item ${\uparrow}a={\uparrow}[a)$, and ${\downarrow}a={\downarrow}(a]$,
for each $a\in L$. 
\item If $F\in\mathrm{Fi}(A)$, then ${\uparrow}$$F\in\mathrm{Fi}(A)$.
\item If $I\in\mathrm{Id}(A)$, then ${\downarrow}I\in\mathrm{Id}(A)$.
\end{enumerate}
\end{lem}

\begin{thm}
\label{teorema1}Let $A\in\mathsf{Sub}$. Let $a\in A$ and $P\in\mathrm{\Ul}(A).$
Then $a\notin{\uparrow}P$ iff there exists $Q\in\Ul(A)$ such that
${\uparrow}P\subseteq Q$ and $a\notin Q$.
\end{thm}

\begin{proof}
See \cite{Celani QUasi modal}.
\end{proof}

Let $A\in\mathsf{Sub}.$ Let $a\in A$. We will denote with ${\downarrow}^{2}a$
to the ideal ${\downarrow}({\downarrow}a)=\bigcup\left\{ {\downarrow}c:c\in{\downarrow}a\right\} $. 

We consider the following conditions defined on $\prec:$ 
\begin{enumerate}
\item[$\mathrm{(S5)}$]  $a\prec b$ implies $a\leq b$.

This condition is equivalent to the condition ${\downarrow}a\subseteq\left(a\right]$,
for $a\in A$.
\item[$\mathrm{(S6)}$]  $a\prec b$ implies that there exists $c\in A\in$ such that $a\prec c\prec b$. 

This condition is equivalent to the condition ${\downarrow}a\subseteq{\downarrow}^{2}a$,
for $a\in A$.
\item[$\mathrm{(S7)}$] $a\prec b$ implies $\neg b\prec\neg a$ . 
\item[$\mathrm{(S8)}$]  $a\neq0\text{ implies }\exists b\neq0\text{ }(b\prec a)$.
\end{enumerate}
Then we define the following subclasses of $\mathsf{Sub}$: 
\begin{enumerate}
\item The class of \emph{topological} subordination algebras, or \emph{S4-subordination}
algebras, $\mathsf{SubS4}=\mathsf{Sub}+\left\{ \mathrm{(S5),(S6)}\right\} $. 
\item The class of \emph{monadic} subordination algebras, or \emph{S5-subordination}
algebras, or quasi-monadic algebras, $\mathsf{SubS5}=\mathsf{SubS4}+\left\{ \mathrm{(S7)}\right\} $.
\item The class of \emph{compingent} algebras $\mathsf{Com}=\mathsf{SubS5}+\left\{ (\mathrm{S8})\right\} $. 
\item The class of de \emph{Vries} algebras $\mathsf{DeV}$, i.e., algebras
$\left<A,\prec\right>\in\mathsf{Com}$ such that $A$ is a complete
Boolean algebra.
\end{enumerate}

\section{Topological representation }\label{sec:Topological-representation}

We recall that a \emph{bitopological} space is a triple $\left<X,\tau,\sigma\right>$
where $\tau$ and $\sigma$ are two topologies on $X$. We write $\left<X,\tau,\tau_{S}\right>$
to denote a bitopological space where $\tau_{S}$ is a topology with
a base $S\subset\mathcal{P}(X)$. The interior and closure of any
subset $Y\subseteq X$ in the topology $\tau_{S}$ will be denoted
by $\mathrm{int}_{S}(Y)$ and $\mathrm{cl}_{S}(Y)$, respectively.

Let $\left<X,\tau,\tau_{S}\right>$ be a bitopological space. Let
$\Clop(X)$ be the Boolean algebra of clopen subsets of $\left<X,\tau\right>$.
Consider a relation $\prec_{S}$ in $\Clop(X)$ defined by 
\[
U\prec_{S}V\text{ sii }U\subseteq\mathrm{int}_{S}(V),
\]
for $U,V\in\Clop(X)$. Then it is immediate to see that $\left<\Clop(X),\prec_{S}\right>$
is a subordination algebra satisfying (S5). 

We define two auxiliary binary relations $\leq_{S}$ and $\triangleleft_{S}$
on $X$ as follows:
\begin{itemize}
\item $x\leq_{S}y$ iff $\forall U\in\Clop(X)\,[x\in\mathrm{int}_{S}(U)$
implies $y\in\mathrm{int}_{S}(U)${]}. 
\item $x\triangleleft_{S}y$ iff $\forall U\in\Clop(X)\,[x\in\mathrm{int}_{S}(U)$
implies $y\in U]$.
\end{itemize}
It is clear that $\leq_{S}$ is reflexive and transitive, $\leq_{S}\subseteq\triangleleft_{S}$,
and that $\triangleleft_{S}$ is reflexive. We note that the relation
$\leq_{S}$ is precisely the specialization order of the topological
space $\left<X,\tau_{S}\right>$.

\begin{defn}
\label{def:S4-space}A bitopological space $\left<X,\tau,\tau_{S}\right>$
is a\emph{ $\mathrm{S4}$-subordination} \emph{space} if: 
\begin{enumerate}
\item[$\mathrm{Sp1}$] $\left<X,\tau\right>$ is a Stone space.
\item[$\mathrm{Sp2}$] \label{Sp2}The family $S=\left\{ \mathrm{int}_{S}(U):U\in\Clop(X)\right\} $
is a base for the topology $\tau_{S}$. 
\item[$\mathrm{Sp3}$] \label{Sp3}$\mathrm{int}_{S}(U)\in\tau$, for each $U\in\Clop(X)$.
 
\item[$\mathrm{Sp4}$] \label{Sp4}$\triangleleft_{S}\subseteq\leq_{S}$. 
\end{enumerate}
\end{defn}

\begin{rem}
Let $\left<X,\tau,\tau_{S}\right>$ be a\emph{ $\mathrm{S4}$}-subordination
space. 

(1) We note that the relation $\triangleleft_{S}$ can be defined
using the map $\varepsilon:X\rightarrow\Ul(\Clop(X))$ defined by
$\varepsilon(x)=\left\{ U\in\Clop(X):x\in U\right\} $ as follows:
\begin{equation}
x\triangleleft_{S}y\text{ iff }{\uparrow}_{S}\varepsilon(x)\subseteq\varepsilon(y),\label{eq:homeomorfismo}
\end{equation}
for all $x,y\in X$, where 
\[
{\uparrow}_{S}\varepsilon(x)=\left\{ V\in\Clop(X):\exists U\in\Clop(X)\,\left(U\prec_{S}V\text{ and }U\in\varepsilon(x)\right)\right\} ,
\]
It is clear that ${\uparrow}_{S}\varepsilon(x)$ is a filter of $\Clop(X)$.

(2) Since $S$ is a base of $\tau_{S}$, we have that the relation
$\leq_{S}$ is the order specialization of $\left<X,\tau_{S}\right>$.

(3) We note that every closed subset of $\left<X,\tau\right>$ is
a compact subset of $\left<X,\tau_{S}\right>$. Indeed. Let $Y$ be
a closed subset in $\left<X,\tau\right>$. Let $C\subseteq\Clop(X)$
and suppose that $Y\subseteq\bigcup\left\{ \mathrm{int}_{S}(U):U\in C\right\} $.
Since $Y$ is closed in $\left<X,\tau\right>$, $Y$ is compact in
$\left<X,\tau\right>$. As $\mathrm{int}_{S}(U)$ is an open of $\left<X,\tau\right>$,
for each $U\in C$, there exists finite subset $\left\{ U_{1},\ldots,U_{n}\right\} $
of $C$ such that $Y\subseteq\mathrm{int}_{S}(U_{1})\cup\ldots\cup\mathrm{int}_{S}(U_{n})$.
Thus, $Y$ is compact in $\left<X,\tau_{S}\right>$. 
\end{rem}

Now we give some equivalent conditions of the condition $\mathrm{(Sp4)}$
of Definition \ref{def:S4-space}.
\begin{prop}
\label{charact S4-subordination spaces}Let $\left<X,\tau,\tau_{S}\right>$
be a bitopological space satisfying the conditions $\mathrm{(Sp1)}$
to \textup{$\mathrm{(Sp3)}$} of Definition \ref{def:S4-space}. Then
the following conditions are equivalent:
\begin{enumerate}
\item For all $U,V\in\Clop(X)$, if $V\prec_{S}U$, then there exists $W\in\Clop(X)$
such that $V\prec_{S}W\prec_{S}U$.
\item ${\downarrow}_{S}U\subseteq{\downarrow}_{S}\left({\downarrow}_{S}U\right)$,
for all $U\in\Clop(X)$.
\item $\triangleleft_{S}\subseteq\leq_{S}$.
\end{enumerate}
\end{prop}

\begin{proof}
It is clear the equivalence $(1)\Leftrightarrow(2)$.  We prove $(2)\Rightarrow(3)$.
Let $x,y\in X$ such that $x\triangleleft_{S}y$. Let $U\in\Clop(X)$.
We suppose that $x\in\mathrm{int}_{S}(U).$ As $\mathrm{int}_{S}(U)$
is an open of the Stone space $\left<X,\tau\right>$, there exists
$W\in\Clop(X)$ such that $W\subseteq\mathrm{int}_{S}(U)$. So, $W\prec_{S}U,$
and thus there exists $V\in\Clop(X)$ such that $W\subseteq\mathrm{int}_{S}(V)$
and $V\subseteq\mathrm{int}_{S}(U)$. As $x\in W\subseteq\mathrm{int}_{S}(V)$,
and $x\triangleleft_{S}y$, we have $y\in V$. Then $y\in\mathrm{int}_{S}(U)$,
and therefore $x\leq_{S}y$. 

We prove $(3)\Rightarrow(2)$. Suppose that there exists $U\in\Clop(X)$
such that $V\in{\downarrow}_{S}U$ and $V\notin{\downarrow}_{S}\left({\downarrow}_{S}U\right)$.
By Lemma \ref{lem:filter and ideal}, ${\downarrow}_{S}\left({\downarrow}_{S}U\right)$
is an ideal of $\Clop(X)$. Then there exists $P\in\Ul(\Clop(X))$
such that $V\in P$ and ${\downarrow}_{S}\left({\downarrow}_{S}U\right)\cap P=\emptyset$.
So, ${\downarrow}_{S}U$$\cap\uparrow_{S}P=\emptyset$, and as $\uparrow_{S}P$
is a filter, there exists $Q\in\Ul(\Clop(X))$ such that $\uparrow_{S}P\subseteq Q$
and ${\downarrow}_{S}U\cap Q=\emptyset$. By Theorem \ref{teorema1},
there exists $D\in\Ul(\Clop(X))$ such that $\uparrow_{S}Q\subseteq Z$,
and $U\notin D$. By Stone duality, there exist $x,y,z\in X$ such
that $P=\varepsilon(x)$, $Q=\varepsilon(y)$ and $D=\varepsilon(z)$.
So $\uparrow_{S}\varepsilon(x)\subseteq\varepsilon(y)$, $\uparrow_{S}\varepsilon(y)\subseteq\varepsilon(z)$.
Then $x\triangleleft_{S}y$ and $y\triangleleft_{S}z$. Since $\leq_{S}=\triangleleft_{S}$,
we have $x\leq_{S}y$ and $y\leq_{S}z$. But as $\leq_{S}$ is transitive,
$x\leq_{S}z$, i.e., $x\triangleleft_{S}z$. Since $x\in V$ and $V\in{\downarrow}_{S}U$,
i.e., $V\subseteq\mathrm{int}_{S}(U)$, we have $z\in U$, which is
a contradiction. Therefore, ${\downarrow}_{S}U\subseteq{\downarrow}_{S}\left({\downarrow}_{S}U\right)$.
\end{proof}

\begin{prop}
Let $\left<X,\tau,\tau_{S}\right>$ be a \emph{$\mathrm{S4}$-subordination}
\emph{space. Then $\left<\Clop(X),\prec_{S}\right>\in\mathsf{SubS4}$. }
\end{prop}

\begin{proof}
Is is immediate to see that $\prec_{S}$ is a subordination defined
on $\Clop(X)$ such that $U\subseteq V$ whenever $U\prec_{S}V$.
By Proposition \ref{charact S4-subordination spaces}, if $V\prec_{S}U$,
then there exists $W\in\Clop(X)$ such that $V\prec_{S}W\prec_{S}U$,
for all $U,V\in\Clop(X)$. Thus, $\left<\Clop(X),\prec_{S}\right>\in\mathsf{SubS4}$. 
\end{proof}

\begin{thm}
\label{prop:propierties of dual space of an alg}Let $\left<A,\prec\right>\in\mathsf{SubS4}$.
Let $\left<X,\tau\right>$ be the Stone space of $A$. Then the family
$S=\left\{ \beta({\downarrow}a):a\in A\right\} $ define a base for
a topology $\tau_{S}$ on $X$ closed under finite intersections such
that $\left<X,\tau,\tau_{S}\right>$ is a $\mathrm{S4}$-subordination
space. Moreover, 
\[
a\prec b\text{ iff }\varphi(a)\prec_{S}\varphi(b),
\]
for each $a,b\in A$.
\end{thm}

\begin{proof}
It is clear that $\beta({\downarrow}(a\wedge b))=\beta({\downarrow}a)\cap\beta({\downarrow}b)$,
for all $a,b\in A$. Moreover, $X=\beta({\downarrow}1)$. Thus $S=\left\{ \beta({\downarrow}a):a\in A\right\} $
define a base for a topology $\tau_{S}$ on $X$ closed under finite
intersections. Since ${\downarrow}a$ is an ideal for each $a\in A$,
we have $\beta({\downarrow}a)$ is an open set of the Stone space
$\left<X,\tau\right>$. Thus $\tau_{S}\subseteq\tau$. 

Now we prove that $\mathrm{int}_{S}(U)\in S$, for all $U\in\Clop(X)$.
Let $U\in\Clop(X)$. By Stone duality, there exists $a\in A$ such
that $U=\varphi(a)$. We prove that 
\begin{equation}
\mathrm{int}_{S}(\varphi(a))=\beta({\downarrow}a).\label{eq:igualdad-1}
\end{equation}
Since $\left<A,\prec\right>$ is reflexive i.e., ${\downarrow}a\subseteq\left(a\right]$
, we have $\beta({\downarrow}a)\subseteq\beta(\left(a\right])=\varphi(a)$.
Taking into account that $\beta({\downarrow}a)$ is an open of $\left<X,\tau_{S}\right>$,
because $S=\left\{ \beta({\downarrow}a):a\in A\right\} $ is a base
of $\left<X,\tau_{S}\right>$, we get 
\[
\mathrm{int}_{S}\beta({\downarrow}a))=\beta({\downarrow}a)\subseteq\mathrm{int}_{S}(\varphi(a)).
\]
We prove the other inclusion. Let $x\in\mathrm{int}_{S}(\varphi(a))=\bigcup\left\{ \beta({\downarrow}b):\beta({\downarrow}b)\subseteq\varphi(a)\right\} $.
Then there exists $b\in A$ such that $x\in\beta({\downarrow}b)\subseteq\varphi(a)=\beta(\left(a\right])$.
So, ${\downarrow}b\subseteq\left(a\right]$, and as $\downarrow$
is monotonic in the lattice of ideals of $A$, we have ${\downarrow}^{2}b={\downarrow}b\subseteq{\downarrow}\left(a\right]={\downarrow}a$,
and since ${\downarrow}b\cap x\neq\emptyset$, we get ${\downarrow}a\cap x\neq\emptyset$,
i.e. $x\in\beta({\downarrow}a)$. So, $\mathrm{int}_{S}(\varphi(a))\subseteq\beta({\downarrow}a)$.
Thus we have proved that $\mathrm{int}_{S}(U)\in S$.

As $\varphi$ is a Boolean isomorphism, we need only prove that $\varphi$
is a subordination homomorphism, i.e.,
\[
a\prec b\text{ iff }\varphi(a)\prec_{S}\varphi(b),
\]
for each $a,b\in A$. Let $a,b\in A$. Suppose that $a\prec b.$ As
$\downarrow b\subseteq\left(b\right]$, we have $\beta(\downarrow b)\subseteq\varphi(b)$,
and thus 
\[
\varphi(a)\subseteq\beta(\downarrow b)\subseteq\bigcup\left\{ \beta(\downarrow c):\beta(\downarrow c)\subseteq\varphi(b)\right\} =\mathrm{int}_{S}(\varphi(b)),
\]
ie., $\varphi(a)\prec_{S}\varphi(b)$. Conversely. If $\varphi(a)\subseteq\mathrm{int}_{S}(\varphi(b))$,
by (\ref{eq:igualdad-1}), we have $\mathrm{int}_{S}(\varphi(b))=\beta(\downarrow b)$.
So, $\varphi(a)\subseteq\beta(\downarrow b)$, i.e., $a\prec b$. 
\end{proof}

\begin{rem}
Let $\left<A,\prec\right>\in\mathrm{SubS4}$ and let $\left<X,\tau,\tau_{S}\right>$
be the $\mathrm{S4}$-subordination space of $\left<A,\prec\right>$.
We note that
\[
x\triangleleft_{S}y\text{ iff }\uparrow x\subseteq y,
\]
for any $x,y\in X$, because by Proposition \ref{prop:propierties of dual space of an alg}
we have:
\[
\begin{array}{ccl}
x\triangleleft_{S}y & \Leftrightarrow & \forall a(x\in\mathrm{int}_{S}(\varphi(a))\Rightarrow y\in\varphi(a)\\
 & \Leftrightarrow & \forall a\exists b(x\in\varphi(b)\subseteq\mathrm{int}_{S}(\varphi(a))=\beta(\downarrow a)\Rightarrow y\in\varphi(a))\\
 & \Leftrightarrow & \forall a\exists b(b\in x\,\&\,b\prec a\Rightarrow a\in y\\
 & \Leftrightarrow & \forall a(b\in{\downarrow}a\cap x\Rightarrow a\in y)\\
 & \Leftrightarrow & \forall a(a\in\uparrow x\Rightarrow a\in y)\\
 & \Leftrightarrow & \uparrow x\subseteq y.
\end{array}
\]
Similarly, we can see that 
\[
x\leq_{S}y\text{ iff }\uparrow x\subseteq\uparrow y.
\]
\end{rem}

\begin{rem}
\label{rem: definition of R}Let $\left<A,\prec\right>\in\mathsf{SubS4}$.
The relation $\triangleleft_{S}\subseteq X\times X$ considered in
our construction coincides with the relation $R$ defined in paragraph
following \cite[Definition 2.13]{Bezhanishvili-Mines-Morandi 2007}
and the relation $R$ defined in \cite[page 725]{Celani QUasi modal}.
However, as already observed in the proof of Theorem \ref{prop:propierties of dual space of an alg},
---the nature of the duality is essentially different. The representations
given in \cite{Bezhanishvili-Mines-Morandi 2007} and \cite{Celani QUasi modal}
are fundamentally relational and “modal-like,” in the sense that it
separates the Stone space from the additional relational structure.
By contrast, the representation given in \ref{prop:propierties of dual space of an alg}
is purely topological: all the algebraic information is captured by
a pair of topologies on the same underlying set. 
\end{rem}

Let $\left<X,\tau,\tau_{S}\right>$ be a $\mathrm{S4}$-subordination
space and let $\left<\Ul(\Clop(X)),\tau_{\varphi},\tau_{\varphi_{S}}\right>$
be the dual space of the topological subordination algebra $\left<\Clop(X),\prec_{S}\right>$.
We recall that the Stone topology $\tau_{\varphi}$ is generated by
the base $\left\{ \varphi(U):U\in\Clop(X)\right\} $, and the topology
$\tau_{\varphi_{S}}$ is generated by the base $S_{\varphi_{S}}=\left\{ \mathrm{int}_{\tau_{\varphi_{S}}}(\varphi(U)):U\in\Clop(X)\right\} $.
Now we prove that the map $\varepsilon:X\rightarrow\Ul(\Clop(X))$
is a bitopological homeomorphism. 
\begin{thm}
\label{lem:caracterizacion order-1}Let $\left<X,\tau,\tau_{S}\right>$
be a $\mathrm{S4}$-subordination space. Then the map $\varepsilon:X\rightarrow\Ul(\Clop(X))$
is a bitopological homeomorphism. 
\end{thm}

\begin{proof}
It is clear that the map $\varepsilon:X\rightarrow\Ul(\Clop(X))$
is an homeomorphism between the Stone spaces $\left<X,\tau\right>$
and $\left<\Ul(\Clop(X)),\tau_{\varphi}\right>$. Thus, we need to
prove only that $\varepsilon$ is an open and continuous map between
the topological spaces $\left<X,\tau_{S}\right>$ and $\left<\Ul(\Clop(X)),\tau_{\varphi_{S}}\right>$.
Since $\varepsilon$ is a bijective map and $S$ is a base, it is
enough to show that $\varepsilon[H]$ is an open subset of $\left<\Ul(\Clop(X)),\tau_{\varphi_{S}}\right>$,
for each $H\in S$. Let $H\in S$ and $x\in X$. Then there exists
$U\in\Clop(X)$ such that $H=\mathrm{int}_{S}U$. So, 
\[
\begin{array}{ccl}
x\in\mathrm{int}_{S}U & \text{iff } & \exists V\in\Clop(X)\,(x\in V\prec_{S}U)\\
 & \text{iff } & \exists V\in\varepsilon(x)\text{ and }\varphi(V)\prec_{S_{\varphi_{S}}}\varphi(U)\\
 & \text{iff } & \varepsilon(x)\in\mathrm{int}_{S_{\varphi_{S}}}\varphi(U).
\end{array}
\]
Then, $\varepsilon[H]=\mathrm{int}_{S_{\varphi_{S}}}\varphi(U)$ and
$\varepsilon$ is an open map. 

Let $B\in S_{\varphi_{S}}$. Then there exist $U\in\Clop(X)$ such
that $B=\mathrm{int}_{S_{\varphi_{S}}}\varphi(U)$. Let $x\in X$.
Then
\[
\begin{array}{ccl}
x\in\varepsilon^{-1}[\mathrm{int}_{S_{\varphi_{S}}}\varphi(U)] & \text{iff} & \varepsilon(x)\in\mathrm{int}_{S_{\varphi_{S}}}\varphi(U)\\
 & \text{iff} & \exists V\in\Clop(X)\,(\varepsilon(x)\in\varphi(V)\prec_{S_{\varphi_{S}}}\varphi(U))\\
 & \text{iff} & \exists V\in\Clop(X)\,(x\in V\prec_{S}U)\\
 & \text{iff} & x\in\mathrm{int}_{S}U.
\end{array}
\]
Thus, $\varepsilon^{-1}[\mathrm{int}_{S_{\varphi_{S}}}\varphi(U)]=\mathrm{int}_{S}U$,
and consequently $\varepsilon$ is a continuous map between the spaces
$\left<X,\tau_{S}\right>$ and $\left<\Ul(\Clop(X)),\tau_{\varphi_{S}}\right>$.
Therefore, $\varepsilon$ is an homeomorphism between the spaces $\left<X,\tau_{S}\right>$
and $\left<\Ul(\Clop(X)),\tau_{\varphi_{S}}\right>$.
\end{proof}

Having identified which are the $S4$-subordination spaces associated
with $S4$-subordination algebras, we will now identify the spaces
associated with the subordination algebras of the classes $\mathrm{SubS5}$,
$\mathsf{Com}$, and $\mathsf{DeV}$.

\begin{prop}
\label{lem:simetria}Let $\left<A,\prec\right>\in\mathsf{SubS4}$
and let $\left<X,\tau,\tau_{S}\right>$ be the $\mathrm{S4}$-subordination
space of $\left<A,\prec\right>$. Then
\end{prop}

\begin{enumerate}
\item The following conditions are equivalent:
\begin{enumerate}
\item $\left<A,\prec\right>\in\mathsf{SubS5}$.
\item $U\subseteq\mathrm{int}_{S}(V)$ iff $\mathrm{cl}_{S}(U)\subseteq V$,
for all $U,V\in\Clop(X).$
\item The relation $\triangleleft_{S}$ is symmetrical.
\end{enumerate}
\item $\left<A,\prec\right>$ satisfies the condition (S8) iff $\mathrm{int}_{S}(U)\neq\emptyset$,
for each $U\in\Clop(X)-\left\{ \emptyset\right\} $.
\end{enumerate}
\begin{proof}
For $\left(1\right),$we note the proof of the equivalence $\left(a\right)\Leftrightarrow\left(b\right)$
is immediate using duality. The equivalence $\left(a\right)\Leftrightarrow\left(c\right)$
follows by the results given in \cite{Celani QUasi modal} (see also
\cite{Bezhanesville-Nick and venema}). The item $\left(2\right)$
is immediate.
\end{proof}

We recall that a topological space $\left<X,\tau\right>$ is said
to be \emph{extremally} \emph{disconnected} if the closure of an open
is open \cite{Kopp}. Also, we recall that a Boolean algebra $B$
is complete iff the dual Stone space of $B$ is extremally disconnected.
\begin{defn}
\label{def:S5. compingent. and de Vries }Let $\left<X,\tau,\tau_{S}\right>$
be a $\mathrm{S4}$-subordination space. We shall say that:
\begin{itemize}
\item $\left<X,\tau,\tau_{S}\right>$ is a $\mathrm{S5}$\emph{-subordination}
space if it is a $\mathrm{S4}$-subordination space such that $U\subseteq\mathrm{int}_{S}(V)$
iff $\mathrm{cl}_{S}(U)\subseteq V$, for all $U,V\in\Clop(X)$.
\item $\left<X,\tau,\tau_{S}\right>$ is a \emph{compingent} space if it
a $\mathrm{S5}$-subordination space such that $\mathrm{int}_{S}(U)\neq\emptyset$,
for each $U\in\Clop(X)-\left\{ \emptyset\right\} $.
\item $\left<X,\tau,\tau_{S}\right>$ is a de \emph{Vries} \emph{space}
if it is a compingent space and $\left<X,\tau\right>$ is extremally
disconnected.
\end{itemize}
\end{defn}

We finish this section with a result that allows another characterization
of Vries spaces.
\begin{lem}
Let $\left<X,\tau,\tau_{S}\right>$ be $\left<X,\tau,\tau_{S}\right>$
is a $\mathrm{S5}$-subordination space such that $\left<X,\tau\right>$
is an extremally disconnected Stone space. Then the following conditions
are equivalent:
\begin{enumerate}
\item $\mathrm{int}_{S}(U)\neq\emptyset$, for all $U\in\Clop(X)-\left\{ \emptyset\right\} $.
\item For each $U\in\Clop(X)-\left\{ \emptyset\right\} $,
\[
U=\mathrm{cl}\left(\bigcup\left\{ V\in\Clop(X):V\subseteq\mathrm{int}_{S}(U)\right\} \right).
\]
\end{enumerate}
\end{lem}

\begin{proof}
We prove $(1)\Rightarrow(2)$ Let $U\neq\emptyset$. Then $\mathrm{int}_{S}(U)\neq\emptyset$.
By $\mathrm{(Sp3)}$of Definition \ref{def:S4-space}, $\mathrm{int}_{S}(U)$
is an open of $\tau$. Since $\Clop(X)$ is a basis of $\left<X,\tau\right>$,
$\mathrm{int}_{S}(U)=\bigcup\left\{ W\in\Clop(X):W\subseteq\mathrm{int}_{S}(U)\right\} .$
So, there exists $W\in\Clop(X)$ such that $W\subseteq\mathrm{int}_{S}(U)$.
Let

Since $H=\bigcup\left\{ V\in\Clop(X):V\subseteq\mathrm{int}_{S}(U)\right\} \subseteq U$,
$\mathrm{cl}(H)\subseteq U$. As $H$ is an open set and $\left<X,\tau\right>$
is extremally disconnected, $\mathrm{cl}(H)$ is an open set, and
thus $\mathrm{cl}(H)\in\Clop(X)$. If $U\nsubseteq\mathrm{cl}(H)$
we get $U\cap\mathrm{cl}(H)^{c}\neq\emptyset$. By hypothesis, of
Def $\mathrm{int}_{S}(U\cap\mathrm{cl}(H)^{c})$$\neq\emptyset$.
So, there exists $W\in\Clop(X)-\left\{ \emptyset\right\} $ such that
\[
W\subseteq\mathrm{int}_{S}(U\cap\mathrm{cl}(H)^{c})\subseteq\mathrm{int}_{S}(U)\cap\mathrm{int}_{S}(\mathrm{cl}(H)^{c}).
\]
As $W\subseteq\mathrm{int}_{S}(U)\subseteq H\subseteq\mathrm{cl}(H)$,
i.e., $W\subseteq\mathrm{cl}(H)$, and as $W\subseteq\mathrm{int}_{S}(\mathrm{cl}(H)^{c})\subseteq\mathrm{cl}(H)^{c}$,
we have $W=\mathrm{cl}(H)\cap\mathrm{cl}(H)^{c}=\emptyset$, which
is a contradiction. So, $U\subseteq\mathrm{cl}(H)$, and thus $\mathrm{cl}(H)=U$.

$\left(2\right)\Rightarrow\left(1\right)$ Let $U\neq\emptyset$.
As $U=\mathrm{cl}\left(\bigcup\left\{ V\in\Clop(X):V\subseteq\mathrm{int}_{S}(U)\right\} \right)$,
there exists $W\subseteq\mathrm{int}_{S}(U)$ such that $W\neq\emptyset$.
Thus, $\mathrm{int}_{S}(U)\neq\emptyset$. 
\end{proof}

\section{Lattice subordinations }\label{sec:Lattice-subordinations}

The following class of subordination algebras was introduced in \cite{Bezhanishvili2013}.
\begin{defn}
Let $\left<A,\prec\right>\in\mathsf{Sub}.$ We shall say that $\prec$
is a lattice subordination if $\prec$ satisfies the following condition:
\begin{enumerate}
\item[$\mathrm{(S9)}$] If $a\prec b$, then there exists $c\in A$ such that $c\prec c$
and $a\leq c\leq b$. 
\end{enumerate}
\end{defn}

Let $\mathsf{SubLat}$ denote class of subordination algebras $\left<A,\prec\right>$
such that $\prec$ lattice subordination. We note if every lattice
subordination yields an $\mathsf{S4}$-subordination algebra (see
\cite[Lemma 2.2]{Bezhanishvili2013}). Hence, $\mathsf{SubLat}\subseteq\mathsf{SubS4}$.
 Moreover, it is easy to see that the set 
\[
B_{\prec}=\left\{ a\in A:a\prec a\right\} 
\]
 is a bounded sublattice of $A$ \cite{Bezhanishvili2013}. In \cite[Section 5]{Bezhanishvili2013}
it was proved that a subordination relation $\prec$ on a Boolean
algebra $A$ is a lattice subordination iff the relation $R_{\prec}\subseteq\Ul(A)\times\Ul(A)$
given by $(P,Q)\in R_{\prec}$ iff ${\uparrow}P\subseteq Q$ is a
Priestley quasi-order \cite{Cignoli -La Falce-Petrovich}. Here we
reinterpreted this result in terms of bitopological duality. Moreover,
the relation $R_{\prec}$ can be defined in terms of $B_{\prec}$,
since
\[
{\uparrow}P\subseteq Q\text{ iff }P\cap B_{\prec}\subseteq Q,
\]
for all $P,Q\in\Ul(A)$.
\begin{defn}
Let $\left<X,\tau,\tau_{S}\right>\in\mathrm{BiSpS4}.$ We shall say
that the topological preorden $\leq_{S}$ is \emph{separated}, if
for each $x,y\in X$ such that $x\nleq_{S}y$ there exists $U\in\Clop(X)$
such that $x\in U$, $y\notin U$, and $U=\mathrm{int}_{S}(U)$.
\end{defn}

\begin{thm}
Let $\left\langle X,\tau,\tau_{S}\right\rangle $ be the dual space
of $\left<A,\prec\right>\in\mathrm{SubS4}$. Then $\left<A,\prec\right>$
satisfies the condition $\mathrm{(S9)}$ iff $\leq_{S}$ is separated. 
\end{thm}

\begin{proof}
$\Rightarrow)$ Assume that $\left<A,\prec\right>$ satisfies the
condition $\mathrm{(S9)}$. Let $x,y\in X$ such that $x\nleq_{S}y$.
Then there exists $U\in\Clop(X)$ such that $x\in\mathrm{int}_{S}(U)$,
and $y\notin\mathrm{int}_{S}(U)$. As $\mathrm{int}_{S}(U)\in\tau$,
$x\in\mathrm{int}_{S}(U)$, and $\Clop(X)$ is a base, there exists
$V\in\Clop(X)$ such that $x\in V\subseteq\mathrm{int}_{S}(U)$. So,
$V\prec_{S}U$, and as $\left<A,\prec\right>$ satisfies (S9), there
exists $c\in A$ such that $c\prec c$ and $V\subseteq\varphi(c)\subseteq U$.
Then 
\[
x\in V\subseteq\varphi(c)=\mathrm{int}_{S}(\varphi(c))\subseteq\mathrm{int}_{S}(U),
\]
 and $y\notin\varphi(c)$. Thus, $\leq_{S}$ is separated. 

$\Leftarrow)$ Assume that $\leq_{S}$ is separated. Let $a,b\in A$
such that $a\prec b$. We need to prove that 
\[
\left[a\right)\cap\left(b\right]\cap B_{\prec}\neq\emptyset.
\]
Suppose the contrary. Let $F$ be the filter generated by $\left[a\right)\cap B_{\prec}$.
We prove that $b\notin F$. Otherwise, there exists $c\in\left[a\right)\cap B_{\prec}$
such that $c\leq b$. But this implies that $a\leq c\leq b$, i.e.,
$c\in\left[a\right)\cap\left(b\right]\cap B_{\prec}$, which is a
contradiction. Thus, there exists $y\in X$ such that 
\[
\left[a\right)\cap B_{\prec}\subseteq y\text{ and }b\notin y.
\]
Let $I$ be the ideal generated by the set $(B_{\prec}-y)\cap\left(b\right]$.
We note that $a\notin I$, because in contrary case there exists $d\in B_{\prec}$,
$d\notin y$ such that $a\leq d$. So, $d\in\left[a\right)\cap B_{\prec}\subseteq y$,
i.e., $d\in y$, which is impossible. Therefore, there exists $x\in X$
such that $a\in x$ and $\left((B_{\prec}-y)\cap\left(b\right]\right)\cap x=\emptyset$.
This implies that $B_{\prec}\cap x\subseteq y$. We note that $\uparrow x\nsubseteq y$,
because $a\in x$, $a\prec b$, and $b\notin y$. As $x\nleq_{S}y$,
and $\leq_{S}$ is separated, there exists $c\in A$ such that $c\prec c$,
$c\in x$ and $c\notin y$. So, $c\in B_{\prec}\cap x\subseteq y$,
which is a contradiction. Then $\left[a\right)\cap\left(b\right]\cap B_{\prec}\neq\emptyset$,
i.e., there exists $c\in A$ such that $c\prec c$ and $a\leq c\leq b$.
\end{proof}

\section{Subordination congruences}\label{sec:Subordination-congruences}

We recall that an open filter in a modal algebra $\left<A,\square\right>$
is a filter $F$ such that if $\square a\in F$ for each $a\in F$.
It is known that the lattice of normal filters of $\left<A,\square\right>$
is isomorphic to the lattice of congruences of $\left<A,\square\right>$
\cite{Sambin}. We recall that a \emph{round} \emph{filter} of a subordination
algebra $\left<A,\prec\right>$ is a filter $F$ of $A$ such that
for each $a\in F$ there exists $b\in A$ such that $b\prec a$ and
$b\in F$. It is clear that the notion of round filter is a generalization
of open filter. 

In \cite{Celani-SubIrreQuasi} was defined the notion of \emph{quasi-modal
congruence} as a Boolean congruence satisfying certain compatibility
property with the quasi-modal operator (see the following Definition
\ref{Def:Congruencia}). In \cite{Celani-SubIrreQuasi} was proved
that the lattice of quasi-modal congruences is isomorphic to the lattice
of round filters. In this section we prove that the lattice of round
filters of a $\mathrm{S4}$-subordination algebra $\left<A,\prec\right>$
are in bijective correspondence with the subsets of $\Ul(A)$ closed
with respect to the topology of Stone $\tau$ and saturated with respect
to the topology $\tau_{S}$. This result generalized the known correspondence
between closed subsets of a Stone space $\left<X,\tau\right>$ and
the congruences of the Boolean algebra $\Clop(X)$.
\begin{defn}
\label{Def:Congruencia}Let $\left<A,\prec\right>\in\mathsf{SubS4}$.
Let $\theta$ be a congruence of $A$. We shall say that $\theta$
is a \emph{subordination congruence}, or a $\prec$\emph{-congruence,
}if $\theta$ satisfies the following condition.
\begin{lyxlist}{00.00.0000}
\item [{SubCon}] \noindent If $c\prec a$ and $\left(a,b\right)\in\theta$,
then there exists $d\in A$ such that $d\prec b$ and $(c,d)\in\theta$. 
\end{lyxlist}
\end{defn}

We note that the notion of homomorphism that corresponds to the notion
of subordination congruence is the notion of strong subordination
homomorphism, called qm-morphism in \cite{Celani QUasi modal}. 

The condition SubCon together with the symmetry of $\theta$ implies
that if $\left(a,b\right)\in\theta$ and $d\prec b$, then there exists
$c\in A$ such that $c\prec a$ and $(c,d)\in\theta$. 

Let $\left<X,\tau,\tau_{S}\right>$ be the $S4$-subordination space
of $\left<A,\prec\right>\in\mathsf{SubS4}$. For each $\tau$-closed
$Y$ we consider the relation 
\[
\theta(Y)=\left\{ \left(a,b\right)\in A\times A:\varphi(a)\cap Y=\varphi(b)\cap Y\right\} .
\]
It is easy to see that $\theta(Y)$ is a Boolean congruence of $A$,
and that every Boolean congruence of $A$ is this form. 
\begin{defn}
Let $\left\langle X,\tau,\tau_{S}\right\rangle $ be a $\mathrm{S4}$-subordination
space. We shall say that a subset $Y$ of $X$ is \emph{saturated
with respect to the topology} $\tau_{S}$ if 
\[
Y=\bigcap\left\{ \mathrm{int}_{S}(U):Y\subseteq U\in\Clop(X)\right\} .
\]
\end{defn}

\begin{thm}
Let $\left\langle X,\tau,\tau_{S}\right\rangle $ be the dual space
of $\left<A,\prec\right>\in\mathrm{SubS4}$. Let $Y\subseteq X$ be
a $\tau$-closed of $\left\langle X,\tau\right\rangle $. Then the
following condition are equivalent 
\begin{enumerate}
\item $\triangleleft(x)=\left\{ y\in X:x\triangleleft y\right\} \subseteq Y$,
for each $x\in Y$. 
\item For every $U,V\in\Clop(X)$, if $\left(U,V\right)\in\theta(Y)$ ,
then $\mathrm{int}_{S}(U)\cap Y=\mathrm{int}_{S}(V)\cap Y$.
\item The relation $\theta(Y)$ satisfies the condition $\mathrm{SubCon}$. 
\item $Y$ \emph{saturated with respect to the topology} $\tau_{S}$.
\end{enumerate}
\end{thm}

\begin{proof}
$\left(1\right)\Rightarrow\left(2\right)$ Let $U,V\in\Clop(X)$ such
that $\left(U,V\right)\in\theta(Y)$. Let $x\in\mathrm{int}_{S}(U)\cap Y$.
Suppose that $x\notin\mathrm{int}_{S}(V)$. Then for all $W\in\Clop(X)$,
and $W\subseteq\mathrm{int}_{S}(V)$, $x\notin W$. Thus, ${\downarrow}_{S}V\cap\varepsilon(x)=\emptyset$,
i.e., $V\notin\uparrow_{S}\varepsilon(x)$. Then there exists $P\in\Ul(\Clop(X))$
such that $\uparrow_{S}\varepsilon(x)\subseteq P$ and $V\notin P$.
Since $\varepsilon$ is surjective, there exists $y\in X$ such that
$P=\varepsilon(y)$, $x\triangleleft y$ and $y\notin V$. But as
$\triangleleft(x)\subseteq Y$, we get $y\in Y$. Since $x\triangleleft y$
and $x\in\mathrm{int}_{S}(U)$, we get $y\in\mathrm{int}_{S}(U)$$\subseteq U$.
Thus, $y\in U\cap Y=V\cap Y$, but this implies that $y\in V$, which
is a contradiction. Thus, $\mathrm{int}_{S}(U)\cap Y\subseteq\mathrm{int}_{S}(V)\cap Y$.
The proof of the other inclusion is similar.

$\left(2\right)\Rightarrow\left(3\right)$ Let $U,V\in\Clop(X)$ such
that $\left(U,V\right)\in\theta(Y)$. Let $W\prec_{S}U$. Then 
\[
W\cap Y\subseteq\mathrm{int}_{S}(U)\cap Y=\mathrm{int}_{S}(V)\cap Y=\bigcup\left\{ H_{i}\in\Clop(X):H_{i}\subseteq\mathrm{int}_{S}(V)\right\} \cap Y.
\]
Then 
\[
W\cap Y\subseteq\bigcup\left\{ H_{i}\in\Clop(X):H_{i}\subseteq\mathrm{int}_{S}(V)\right\} ,
\]
 and as $W\cap Y$ is closed, by compactness, there are $H_{1},\ldots,H_{n}\in\Clop(X)$
such that $W\cap Y\subseteq H_{1}\cup\ldots\cup H_{n}\subseteq\mathrm{int}_{S}(V)$.
Let $H=H_{1}\cup\ldots\cup H_{n}$. So, $W\cap Y=W\cap H\cap Y$,
and $W\cap H\subseteq H\prec_{S}V$. Thus there exists $K=W\cap H\in\Clop(X)$
such that $K\prec_{S}V$ and $\left(W,K\right)\in\theta(Y)$.

$\left(3\right)\Rightarrow\left(4\right)$ Let $Y\subseteq U$. We
need to prove that $Y\subseteq\mathrm{int}_{S}(U)$. As $\left(U,X\right)\in\theta(Y)$,
and $X\prec_{S}X$, there exists $V\subseteq\mathrm{int}_{S}(U)$
such that $\left(V,X\right)\in\theta(Y)$. So, $Y\cap V=Y\cap X=Y$,
i.e., $Y\subseteq V\subseteq\mathrm{int}_{S}(U)$. Thus, $Y\subseteq\bigcap\left\{ \mathrm{int}_{S}(U):Y\subseteq U\in\Clop(X)\right\} $.
Let $x\in\bigcap\left\{ \mathrm{int}_{S}(U):Y\subseteq U\in\Clop(X)\right\} $.
If $x\notin Y$, there exists $U\in\Clop(X)$ such that $Y\subseteq U$
and $x\notin U$. But by hypothesis, $x\in\mathrm{int}_{S}(U)\subseteq U$,
which ia a contradiction. Thus, $Y=\bigcap\left\{ \mathrm{int}_{S}(U):Y\subseteq U\in\Clop(X)\right\} $. 

$\left(4\right)\Rightarrow\left(1\right)$ Let $x,y\in X$ and $y\in\triangleleft(x)$.
Suppose that $y\notin Y$. As $Y$ is a $\tau$-closed, there exists
$U\in\Clop(X)$ such that $Y\subseteq U$ and $y\notin U$. So, $Y\subseteq\mathrm{int}_{S}(U)$
and thus $y\notin\mathrm{int}_{S}(U)$. As $x\in\mathrm{int}_{S}(U)$
and $y\notin\mathrm{int}_{S}(U)$, we get that $x\ntriangleleft y$,
which is a contradiction. Therefore, $y\in Y$. 
\end{proof}

\begin{cor}
Let $\left\langle X,\tau,\tau_{S}\right\rangle $ be a bitopological
subordination space. Then the lattice of $\prec$\emph{-congruences
of $\left(\Clop(X),\prec_{S}\right)$ is isomorphic to the lattice
of subsets $X$ such that are closed subsets of the Stone space $\left(X,\tau\right)$
and saturated with respect to the topology} $\tau_{S}$.
\end{cor}

\section{Categorical dualities I}\label{sec:Categorical-dualities-I}

In this section we study morphisms between S4-subordination algebras
based in Boolean homomorphisms. In the next section we will study
morphism but based in meet-homomorphisms. 

We recall that a Boolean homomorphism $h:A_{1}\rightarrow A_{2}$
between modal algebras is a weak modal homomorphism if $h(\square a)\leq\square h(a)$,
for all $a\in A_{1}$, and $h$ is a modal homomorphism if $h(\square a)=\square h(a)$,
for all $a\in A_{1}$. We introduce adequate generalization of the
notion of weak modal homomorphism and modal homomorphism. 
\begin{defn}
\label{def:sub homo}Let $A_{1}=\left<A_{1},\prec_{1}\right>$ and
$A_{2}=\left<A_{2},\prec_{2}\right>$ be two subordination algebras. 
\begin{enumerate}
\item A \emph{subordination} \emph{homomorphism} from $A_{1}$ to $A_{2}$
is a Boolean homomorphism $h:A_{1}\rightarrow A_{2}$ such that for
all $a,b\in A_{1},$ if $a\prec_{1}b$, then $h(a)\prec_{2}h(b)$. 
\item A \emph{strong subordination homomorphism} from $A_{1}$ to $A_{2}$
is a subordination homomorphism $h:A_{1}\rightarrow A_{2}$ satisfying
the additional condition
\begin{enumerate}
\item[(QH)]  If $a\prec_{2}h(b)$, then there exists $c\in A_{1}$ such that
$a\leq h(c)$ and $c\prec_{1}b$.
\end{enumerate}
\end{enumerate}
A \emph{subordination} \emph{isomorphism} is a bijective subordination
homomorphism $h:A_{1}\rightarrow A_{2}$ such that if $h(a)\prec_{2}h(b)$
implies that $a\prec_{1}b$, for $a,b\in A$.
\end{defn}

We note that a Boolean homomorphism $h:A_{1}\rightarrow A_{2}$ is
a subordination homomorphism iff $h\left[{\downarrow}a\right]\subseteq{\downarrow}h(a)$,
for each $a\in A_{1}$. A subordination homomorphism $h:A_{1}\rightarrow A_{2}$
is a strong subordination iff ${\downarrow}h(a)\subseteq\left(h\left[{\downarrow}a\right]\right]$,
for each $a\in A_{1}$. Thus, a Boolean homomorphism $h:A_{1}\rightarrow A_{2}$
is a strong subordination homomorphism iff $\left(h\left[{\downarrow}a\right]\right]={\downarrow}h(a)$,
for each $a\in A_{1}$. We note that the strong subordination homomorphisms
are called qm-homomorphism in \cite{Celani QUasi modal}.

Let $\left<X_{1},\tau_{1},\tau_{S_{1}}\right>$ and $\left<X_{2},\tau_{2},\tau_{S_{2}}\right>$
be two $S4$-subordination spaces. A function is $\tau$-continuous
($\tau_{S}$-continuous) if it is a continuous function $f:X_{1}\rightarrow X_{2}$
between the spaces $\left<X_{1},\tau_{1}\right>$ and $\left<X_{2},\tau_{2}\right>$
(between the spaces $\left<X_{1},\tau_{S_{1}}\right>$ and $\left<X_{2},\tau_{S_{2}}\right>$).
Now we shall study when a $\tau$-continuous function $f:X_{1}\rightarrow X_{2}$
is $\tau_{S}$-continuous. 

Let $\left<X_{1},\tau_{1}\right>$ and $\left<X_{2},\tau_{2}\right>$
be Stone spaces. We recall that a function $f:X_{1}\rightarrow X_{2}$
is a $\tau$-continuous iff the map $f^{*}:\mathrm{Clop}(X_{2})\rightarrow\mathrm{Clop}(X_{1})$
given by $f^{*}(U)=f^{-1}(U),$ for each $U\in\mathrm{Clop}(X_{2})$
is a Boolean homomorphism. 

In the next result we are going to characterize the continuous functions
that are dual to the subordination homomorphism.
\begin{prop}
\label{prop:f stable}Let $f:X_{1}\rightarrow X_{2}$ be a $\tau$-continuous
function between the $S4$-subordination spaces $\left<X_{1},\tau_{1},\tau_{S_{1}}\right>$
and $\left<X_{2},\tau_{2},\tau_{S_{2}}\right>$. The following conditions
are equivalent:
\begin{enumerate}
\item For all $x,y\in X_{1},$ if $x\triangleleft_{1}y$ , then $f(x)\triangleleft_{2}f(y)$,
i.e. $f$ is stable. 
\item $f^{-1}\mathrm{(\mathrm{int}_{S_{2}}}(U))\subseteq\mathrm{int}_{S_{1}}(f^{-1}(U))$,
for each $U\in\mathrm{Clop}(X_{2})$. 
\item $f^{*}$ is a subordination homomorphism.
\end{enumerate}
\end{prop}

\begin{proof}
$\left(1\right)\Rightarrow\left(2\right)$. Let $f(x)\in\mathrm{\mathrm{int}_{S_{2}}}(U)$,
for some $U\in\mathrm{Clop}(X_{2})$. Suppose that $x\notin\mathrm{int}_{S_{1}}(f^{-1}(U))$.
Then it is easy to see that ${\downarrow}_{S_{1}}f^{-1}(U)\cap\varepsilon_{1}(x)=\emptyset$.
So there exists $P\in\Ul(\mathrm{Clop}(X_{1}))$ such that $\uparrow_{S_{1}}\varepsilon_{1}(x)\subseteq P$
and $f^{-1}(U)\notin P$. As $\left<X_{1},\tau_{1}\right>$ is a Stone
space, there exists $y\in X_{1}$ such that $P=\varepsilon_{1}(y)$.
Then $\uparrow_{S_{1}}\varepsilon_{1}(x)\subseteq\varepsilon_{1}(y)$.
By \ref{eq:homeomorfismo}, $x\triangleleft_{S_{1}}y$, and as $f$
is stable, $f(x)\triangleleft_{2}f(y)$. Since $f(x)\in\mathrm{\mathrm{int}_{S_{2}}}(U)$,
we get $f(y)\in\mathrm{int}_{S_{2}}(U)\subseteq U$, which is a contradiction.
Thus, $f^{-1}\mathrm{(\mathrm{int}_{S_{2}}}(U))\subseteq\mathrm{int}_{S_{1}}(f^{-1}(U))$. 

$\left(2\right)\Rightarrow\left(1\right)$ Let $x,y\in X_{1}$. Assume
that $f(x)\ntriangleleft_{2}f(y)$. Then there exists $U\in\mathrm{Clop}(X_{2})$
such that $f(x)\in\mathrm{int}_{S_{2}}(U)$ and $f(y)\notin\mathrm{int}_{S_{2}}(U)$.
Then $x\in f^{-1}(\mathrm{int}_{S_{2}}(U))\subseteq\mathrm{int}_{S_{1}}(f^{-1}(U))$
and $y\notin\mathrm{int}_{S_{1}}(f^{-1}(U))$. Thus, $x\ntriangleleft_{1}y$. 

We prove $\left(2\right)\Rightarrow\left(3\right).$ Let $U,V\in\mathrm{Clop}(X_{2})$.
Suppose that $U\prec_{S_{2}}V$, i.e., $U\subseteq\mathrm{int}_{S_{2}}(V)$.
Then $f^{-1}(U)\subseteq f^{-1}(\mathrm{int}_{S_{2}}(V))\subseteq\mathrm{int}_{S_{1}}(f^{-1}(V)),$
i.e., $f^{-1}(U)\prec_{S_{1}}f^{-1}(V)$.

$\left(3\right)\Rightarrow\left(2\right)$ Let $U\in\mathrm{Clop}(X_{2})$.
Let $x\in f^{-1}\mathrm{(\mathrm{int}_{S_{2}}}(U))$. So, $f(x)\in\mathrm{\mathrm{int}_{S_{2}}}(U)$,
i.e., there exists $W\in\Clop(X_{2})$ such that $W\prec_{S_{2}}U$
and $f(x)\in W$. Then $f^{-1}(W)\prec_{S_{1}}f^{-1}(U),$ i.e., $f^{-1}(W)\subseteq\mathrm{int}_{S_{1}}(f^{-1}(U))$.
Thus, $x\in\mathrm{int}_{S_{1}}(f^{-1}(U))$.
\end{proof}

Now we will study the dual morphisms of strong subordination homomorphisms. 

Let $\left<X,\tau,\tau_{S}\right>$ be a $\mathrm{S4}$-subordination
space. We recall that $\leq_{S}=\triangleleft_{S}$, and that $\leq_{S}$
is the specialization order of the space $\left<X,\tau_{S}\right>$.
Then for each $U\subseteq X$ we consider the downset generated by
$U$ in the poset $\left<X,\triangleleft_{S}\right>$, i.e. $\left(U\right]_{\triangleleft_{S}}=\left\{ y\in X:\exists z\in U\,(y\triangleleft_{S}z)\right\} $.

The following proposition is possibly known, but we give a proof for
completeness
\begin{prop}
\label{prop:closure  increasing-1}Let $\left<X,\tau,\tau_{S}\right>$
be a $\mathrm{S4}$-subordination space. Then $\mathrm{cl}_{S}(U)=\left(U\right]_{\triangleleft}$,
for each $U\in\Clop(X)$.
\end{prop}

\begin{proof}
Let $U\in\Clop(X)$. For each $x\in\mathrm{cl}_{S}(U)$ we take the
set 
\[
F_{x}=\left\{ V\in\Clop(X):x\in\mathrm{int}_{S}(V)\right\} .
\]
We prove that $\bigcap F_{x}\cap U\neq\emptyset$. Otherwise, $U\subseteq\bigcup\left\{ V^{c}:V\in F_{x}\right\} $,
and as $U^{c}$ is compact, there exists a finite subset $\left\{ V_{1},\ldots,V_{n}\right\} \subseteq F_{x}$
such that $U\cap V_{1}\cap\ldots\cap V_{n}=\emptyset$. So, $V_{1}\cap\ldots\cap V_{n}\subseteq U^{c}$,
and thus 
\[
\mathrm{int}_{S}(V_{1}\cap\ldots\cap V_{n})=\mathrm{int}_{S}(V_{1})\cap\ldots\cap\mathrm{int}_{S}(V_{n})\subseteq\mathrm{int}_{S}(U^{c})=\mathrm{cl}_{S}(U)^{c}.
\]
As $x\in\mathrm{int}_{S}(V_{1})\cap\ldots\cap\mathrm{int}_{S}(V_{n})$,
we get $x\in\mathrm{cl}_{S}(U)^{c}$, which is impossible. Thus, there
exists $y\in\bigcap F_{x}\cap U$. Now we prove that $\uparrow_{S}(\varepsilon(x))\subseteq\varepsilon(y)$$.$
Let $U\in\uparrow_{S}(\varepsilon(x))$, i.e., ${\downarrow}_{S}(U)\cap\varepsilon(x)\neq\emptyset$.
So there exists $V\in\Clop(X)$ such that $V\subseteq\mathrm{int}_{S}(U)$
and $x\in V$. Then, $U\in F_{x}$, and thus $y\in U$. So, $\uparrow_{S}(\varepsilon(x))\subseteq\varepsilon(y)$,
i.e., $x\triangleleft_{S}y$. As $y\in U$, we have that $x\in\left(U\right]_{\triangleleft_{S}}$.
Thus, $\mathrm{cl}_{S}(U)\subseteq\left(U\right]_{\triangleleft_{S}}$.

Let $x\in\left(U\right]_{\triangleleft_{S}}.$ Then there exists $y\in X$
such that $x\triangleleft_{S}y$ and $y\in U$. If $x\notin\mathrm{cl}_{S}(U)$,
as the family $S$ is a base of $\tau_{S}$, exists $V\in\Clop(X)$
such that $\mathrm{cl}_{S}(U)\subseteq\mathrm{int}_{S}(V)^{c}$ and
$x\in\mathrm{int}_{S}(V)$. Since $x\triangleleft_{S}y$, we get $y\in\mathrm{int}_{S}(V)$,
but as $y\in U\subseteq\mathrm{cl}_{S}(U)\subseteq\mathrm{int}_{S}(V)^{c}$,
we have $y\in\mathrm{int}_{S}(V)^{c}$, which is a contradiction.
Therefore, $\left(U\right]_{\triangleleft_{_{S}}}\subseteq\mathrm{cl}_{S}(U)$.
\end{proof}

\begin{defn}
Let $f:X_{1}\rightarrow X_{2}$ be a $\tau$-continuous function between
two $S4$-subordination spaces $\left<X_{1},\tau_{1},\tau_{S_{1}}\right>$.
We shall say that $f$ is $\triangleleft$\emph{-stable}, or \emph{stable},
if $f$ satisfies any of the conditions of Proposition \ref{prop:f stable}.
\end{defn}

\begin{prop}
\label{prop: f propiedad (PM)}Let $f:X_{1}\rightarrow X_{2}$ a $\tau$-continuous
stable function between the $\mathrm{S4}$-subordination spaces $\left<X_{1},\tau_{1},\tau_{S_{1}}\right>$
and $\left<X_{2},\tau_{2},\tau_{S_{2}}\right>$. Then the following
condition are equivalent
\begin{enumerate}
\item $f^{*}$ satisfies the property (QH).
\item If $f(x)\triangleleft_{S_{2}}y$, then there exists $z\in X_{1}$
such that $x\triangleleft_{S_{1}}z$ and $f(z)=y$.
\item $\mathrm{int}_{S_{1}}(f^{-1}(U))\subseteq f^{-1}(\mathrm{int}_{S_{2}}(U))$,
for all $U\in\mathrm{\mathrm{Clop}}(X_{2})$.
\end{enumerate}
\end{prop}

\begin{proof}
$\left(1\right)\Rightarrow\left(2\right)$ Let $x\in X_{1}$ and $y\in X_{2}$,
such that $f(x)\triangleleft_{S_{2}}y$, i.e., $\uparrow_{S_{2}}(\varepsilon_{2}(f(x)))\subseteq\varepsilon_{2}(y)$.
We prove that 
\[
\bigcap\left\{ V:{\downarrow}_{S_{1}}V\cap\varepsilon_{1}(x)\neq\emptyset\right\} \cap\bigcap\left\{ f^{-1}(U):y\in U\in\Clop(X_{2})\right\} \neq\emptyset.
\]
Otherwise, by compacity, there exists a finite set $\left\{ V_{1},\ldots,V_{n},U_{1},\ldots,U_{k}\right\} $
such that ${\downarrow}_{S_{1}}(V_{1}\cap\ldots\cap V_{n})\cap\varepsilon_{1}(x)\neq\emptyset$,
$y\in U_{1}\cap\ldots\cap U$, and 
\[
V_{1}\cap\ldots\cap V_{n}\cap f^{-1}(U_{1})\cap\ldots\cap f^{-1}(U_{k})=\emptyset.
\]
Then $V=V_{1}\cap\ldots\cap V_{n}\subseteq f^{-1}(U^{c}_{1}\cup\ldots\cup U^{c}_{k})=f^{-1}(U^{c})$,
where $U=U_{1}\cap\ldots\cap U_{k}$. This implies that ${\downarrow}_{S_{1}}(V)\subseteq{\downarrow}_{S_{1}}(f^{-1}(U^{c}))$,
and so ${\downarrow}_{S_{1}}(f^{-1}(U^{c}))\cap\varepsilon_{1}(x)\neq\emptyset$.
Then there exists $W\in\varepsilon_{1}(x)$ such that $W\subseteq\mathrm{int}_{S_{1}}(f^{-1}(U^{c})).$
So, $W\prec_{S_{1}}f^{-1}(U^{c})$, and by the property (QH) there
exists $N\in\Clop(X_{2})$ such that $W\subseteq f^{-1}(N)$ and $N\prec_{S_{2}}U^{c}$.
Then, $N\in\varepsilon_{2}(f(x))\cap{\downarrow}_{S_{2}}(U^{c})$,
i.e., $U^{c}\in\uparrow_{S_{2}}(\varepsilon_{2}(f(x)))$, and by hypothesis,
$U^{c}\in\varepsilon_{2}(y)$, i.e., $y\in U^{c}$, which is impossible.
Therefore, there exists 
\[
z\in\bigcap\left\{ V:\downarrow_{S_{1}}(V)\cap\varepsilon_{1}(x)\neq\emptyset\right\} \cap\bigcap\left\{ f^{-1}(U):y\in U\in\Clop(X_{2})\right\} .
\]
Now, it is easy to see that $\uparrow_{S_{1}}(\varepsilon_{1}(x))\subseteq\varepsilon_{1}(z)$
and $f(z)=y$. So, $x\triangleleft_{S_{1}}z$ and $f(z)=y$.

We prove $\left(2\right)\Rightarrow\left(3\right)$. Let $U\in\mathrm{\mathrm{Clop}}(X_{2})$
and $x\in\mathrm{int}_{S_{1}}(f^{-1}(U))$. Suppose that $x\notin f^{-1}(\mathrm{int}_{S_{2}}(U))$,
i.e., $f(x)\in\mathrm{cl}_{S_{2}}(U^{c}).$ By Proposition \ref{prop:closure  increasing-1},
$\mathrm{cl}_{S_{2}}(U^{c})=\left(U^{c}\right]_{S_{2}}$. Then there
exists $y\in X_{2}$ such that $f(x)\triangleleft_{S_{2}}y$ and $y\notin U$.
So, there exists $z\in X_{1}$ such that $x\triangleleft_{S_{1}}z$
and $f(z)=y$. As $x\in\mathrm{int}_{S_{1}}(f^{-1}(U))$, $z\in\mathrm{int}_{S_{1}}(f^{-1}(U))\subseteq f^{-1}(U)$,
i.e., $f(z)=y\in U$, which is impossible. Therefore, $\mathrm{int}_{S_{1}}(f^{-1}(U))\subseteq f^{-1}(\mathrm{int}_{S_{2}}(U))$.

$\left(3\right)\Rightarrow\left(1\right)$. Let $U\in\mathrm{Clop}(X_{1})$
and $V\in\mathrm{Clop}(X_{2})$ such that $U\prec_{S_{1}}f^{-1}(V)$.
Then $U\subseteq\mathrm{int}_{S_{1}}(f^{-1}(V))$, and so $U\subseteq f^{-1}(\mathrm{int}_{S_{2}}(V))$.
It is easy to see that $f^{-1}(\mathrm{int}_{S_{2}}(V))=\bigcup\left\{ f^{-1}(W):W\in\mathrm{Clop}(X_{2})\text{ and }W\subseteq\mathrm{int}_{S_{2}}(V)\right\} $.
Since $U$ is compact, there is $W\in\mathrm{Clop}(X_{2})$ such that
$U\subseteq f^{-1}(W)$ and $W\subseteq\mathrm{int}_{S_{2}}(V)$. 

\end{proof}

\begin{cor}
Let $f:X_{1}\rightarrow X_{2}$ be a bijective $\tau$-continuous
function between two $S4$-subordination spaces. Then the following
conditions are equivalent:
\begin{enumerate}
\item $x\triangleleft_{1}y$ iff $f(x)\triangleleft_{2}f(y)$, for all $x,y\in X_{1}$.
\item $f^{-1}(U)\in\tau_{S_{2}}$, for any $U\in\tau_{S_{1}}$. Thus,$f$
is $\tau_{S}$-continuous. 
\end{enumerate}
\end{cor}

\begin{defn}
\label{def:stable morphism}Let $f:X_{1}\rightarrow X_{2}$ be a $\tau$-continuous
function between two $S4$-subordination spaces $\left<X_{1},\tau_{1},\tau_{S_{1}}\right>$.
We shall say that $f$ is strong $\triangleleft$\emph{-stable}, or
\emph{strong stable}, if $f$ is stable and satisfies any of the conditions
of Proposition \ref{prop: f propiedad (PM)}.
\end{defn}

For $i=1,2$, let $\left<A_{i},\prec_{i}\right>\in\mathsf{SubS4}$,
and let $\left<X_{i},\tau_{i},\tau_{S_{i}}\right>$ its dual spaces.
If $h:A_{1}\rightarrow A_{2}$ is a Boolean homomorphism then we define
the map $h_{*}:X_{2}\rightarrow X_{1}$ by $h_{*}(y)=h^{-1}(y)$,
for each $y\in X_{2}$. By Boolean duality, Proposition \ref{prop:f stable},
and Proposition \ref{prop: f propiedad (PM)} we have the following
characterization of subordination and strong subordination homomorphisms.
\begin{cor}
\label{thm:duality homomo}Let $\left<A_{i},\prec_{i}\right>\in\mathsf{SubS4}$,
and let $\left<X_{i},\tau_{i},\tau_{S_{i}}\right>$ its dual spaces.
Let $h:A_{1}\rightarrow A_{2}$ be a Boolean homomorphism. Then
\begin{enumerate}
\item $h$ is a subordination homomorphism iff $h_{*}:X_{2}\rightarrow X_{1}$
is stable. 
\item $h$ is a strong subordination homomorphism iff $h_{*}:X_{2}\rightarrow X_{1}$
is strong stable. 
\end{enumerate}
\end{cor}

\begin{proof}
It follows by Proposition \ref{prop:f stable} and Proposition \ref{prop: f propiedad (PM)}.
\end{proof}

It is clear that the usual composition of stable or strong stable
morphisms is a stable or strong stable morphism. Thus, for each class
of $\mathrm{S4}$-subordination spaces studied in this paper there
exists two categories, depending on whether we use the notion of stable
morphism or the notion of strong stable morphism.We consider the categories
$\mathbf{SubS4}_{h}$ and $\mathbf{SubS4}_{st}$ whose objects belong
to $\mathrm{SubS4}$ and whose morphisms are subordination homomorphisms,
and strong subordination homomorphisms, respectively. Let $\mathrm{BiSpS4}$
be the class of $\mathrm{S4}$-subordination spaces. We consider the
categories $\mathbf{BiSpS4}_{s}$and $\mathbf{BiSpS4}_{st}$ whose
morphisms are stable morphisms, and strong stable morphisms, respectively.

By Propositions \ref{prop:propierties of dual space of an alg}, \ref{lem:caracterizacion order-1},
\ref{prop:f stable} and \ref{prop: f propiedad (PM)} we have the
following result.
\begin{thm}
\label{thm: duality categorical I}The categories $\mathbf{SubS4}_{h}$
and $\mathbf{SubS4}_{st}$\textup{ are dually equivalent to the categories
}$\mathbf{BiSpS4}_{s}$ and $\mathbf{BiSpS4}_{st}$ , respectively.
\end{thm}

\begin{rem}
The duality established in Theorem \ref{thm: duality categorical I}
for S4\nobreakdash-subordination algebras is closely related to the
duality presented in \cite[Thm. 4.11(2)]{Bezhanesville-Nick and venema}
and the relational representation given in \cite[Theorem 26]{Celani QUasi modal}.
In these works, a S4\nobreakdash-subordination algebra $(B,\prec)$
is represented by a pair $(X,R)$ where $X$ is the Stone space of
$B$ and $R$ is a closed quasi\nobreakdash-order (i.e., a reflexive
and transitive closed relation) on $X$. The subordination $\prec$
is then recovered by the relation $R$ as follows: $a\prec b\iff R[\varphi(a)]\subseteq\varphi(b)$,
for all $a,b\in B$, where $\varphi$ denotes the Stone map. In our
approach, we instead use a bitopological space $(X,\tau,\tau_{S})$
where $\tau$ is the Stone topology of $B$ and $\tau_{S}$ is an
additional topology whose opens are the sets $\Int_{S}(U)$ for clopen
$U$ of $X$. The relation $R$ considered in \cite{Celani QUasi modal,Bezhanesville-Nick and venema}
coincides with the specialization order of $\tau_{S}$. Moreover,
the condition $R[\varphi(a)]\subseteq\varphi(b)$ becomes $\varphi(a)\subseteq\Int_{S}(\varphi(b))$,
i.e., $\varphi(a)\prec_{S}\varphi(b)$. Thus, the two dualities are
essentially equivalent at the level of object, but they differ in
the choice of morphisms and in the underlying geometric picture: \cite{Bezhanesville-Nick and venema}
works with stable continuous maps between relational Stone spaces,
while we work with bitopological continuous maps (with respect to
both topologies). The bitopological perspective offers a direct link
to the theory of bitopological spaces and allows a natural treatment
of round filters and congruences (see Sections\,\ref{sec:Subordination-congruences}).
Consequently, Theorem\,\ref{thm: duality categorical I} can be seen
as a bitopological reinterpretation of the relational duality given
\cite{Bezhanesville-Nick and venema} and \cite{Celani QUasi modal}.
\end{rem}

\section{Categorical Dualities II}\label{sec:Subordination-meet-hemimorphisms}

Under the Vries duality, the dual morphisms of continuous functions
between compact Hausdorff spaces are meet-homomorphisms between de
Vries algebras satisfying additional conditions \cite{de Vries tesis}.
Thus, it is natural consider morphisms between S4-subordination algebras
based on meet-homomorphisms rather than of Boolean homomorphisms.
In this section, we explore different alternatives for morphisms between
S4-subordination algebras and characterize these alternatives in topological
terms.

\medskip{}

Let $A_{1}$ and $A_{2}$ be two Boolean algebras. A \emph{meet-homomorphism,}
between $A_{1}$ and $A_{2}$ is a function $h:A_{1}\rightarrow A_{2}$
such that $h(1)=1$ and $h(a\wedge b)=h(a)\wedge h(b)$, for all $a,b\in A_{1}$.
The dual map $g:A_{1}\rightarrow A_{2}$ is defined as $g(a)=\neg h(\neg a),$
for each $a\in A_{2}$. It is clear that $g$ is a join-homomorphism,
i.e., $g(a\vee b)=g(a)\vee g(b)$ and $g(0)=0$, for all $a,b\in A_{1}$. 

The following result is known and will be used later.
\begin{lem}
Let $A_{1}$ and $A_{2}$ be two Boolean algebras. A\emph{ map} $h:A_{1}\rightarrow A_{2}$
is a meet-homomorphism iff $h^{-1}(F)\in\mathrm{Fi}(A_{1})$, for
every $F\in\mathrm{Fi}(A_{2})$.
\end{lem}

Let $X_{1}$ and $X_{2}$ be two Stone spaces. A \emph{Boolean} \emph{relation}
is a binary relation $R\subseteq X_{1}\times X_{2}$ satisfying the
conditions:
\begin{enumerate}
\item $R(x)$ is a closed subset for each $x\in X_{1}$,
\item $h_{R}(U)=\left\{ x\in X_{1}:R(x)\subseteq U\right\} \in\mathrm{Clop}(X_{1})$,
for each $U\in\mathrm{Clop}(X_{2})$. 
\end{enumerate}
It is knows that there exists a duality between meet-homomorphism
of Boolean algebras and Boolean relations (see \cite{Halmos1962},
\cite{Kopp} \cite{Sambin}, and \cite{Wright}). If $h:A_{1}\rightarrow A_{2}$
is a meet-homomorphism, and $X_{1}$ be the Stone space of $A_{1}$
and $X_{2}$ be the Stone space of $A_{2}$, then the relation $R_{h}\subseteq X_{2}\times X_{1}$
defined by 
\[
(x,y)\in R_{h}\text{ iff }h^{-1}\left[x\right]\subseteq y,
\]
is a Boolean relation. We recall that for each $a\in A_{1}$, $h(a)\notin x$
iff there exists $y\in X_{2}$ such that $(x,y)\in R_{h}$ and $a\notin y$. 

Conversely, given a Boolean relation $R$ between the Stone spaces
$X_{1}$ and $X_{2}$, we define the function
\[
h_{R}:\mathrm{Clop}(X_{2})\rightarrow\mathrm{Clop}(X_{1})
\]
 by setting 
\[
h_{R}(U)=\left\{ x\in X_{1}:R(x)\subseteq U\right\} .
\]
It is easy to see that function $h_{R}$ is a meet-homomorphism. The
function $g_{R}:\mathrm{Clop}(X_{2})\rightarrow\mathrm{Clop}(X_{1})$
defined by 
\[
g_{R}(U)=\left\{ x\in X_{1}:R(x)\cap U\neq\emptyset\right\} 
\]
is a join-homomorphism. If $h:A_{1}\rightarrow A_{2}$ is a meet-homomorphism,
and $R=R_{h}$ is the Boolean relation between $X_{2}$ and $X_{1}$,
then $h_{R}(\varphi(a))=\varphi(h(a))$ for all $a\in A_{1}$.

In the next results we assume that $\left<A,\prec_{1}\right>$ and
$\left<A_{2},\prec_{2}\right>$ are S4-subordination algebras and
$\left<X_{1},\tau_{1},\tau_{S_{1}}\right>$ and $\left<X_{2},\tau_{2},\tau_{S_{2}}\right>$
are the duals S4-subordination spaces of $\left<A_{1},\prec_{1}\right>$
and $\left<A_{2},\prec_{2}\right>$, respectively. We also suppose
$h:A_{1}\rightarrow A_{2}$ is a meet-homomorphism and $R\subseteq X_{2}\times X_{1}$
is the dual Boolean relation of $h$. 

Now we are going to study certain properties required in meet-homomorphisms
that are used to define different classes of morphisms between subordination
algebras and its topological characterizations. 

Let us consider the following conditions defined on a meet-homomorphism
$h:A_{1}\rightarrow A_{2}$:
\begin{enumerate}
\item[(CO)]  $h(0)=0$.
\item[(MO)]  If $a\prec b$ implies $h(a)\prec h(b)$.
\item[(QH)]  If $a\prec h(b)$, then there exists $c\in A$ such that $a\leq h(c)$
and $c\prec b$.
\item[(DH)]  If $a\prec b$, then $g(a)\prec h(b),$for all $a,b\in A_{1}$.
\end{enumerate}
The condition (MO) can be write as $h\left[{\downarrow}a\right]\subseteq{\downarrow}h(a)$,
for all $a\in A_{1}$, and the condition (QH) can be write as ${\downarrow}h(a)\subseteq\left(h\left[{\downarrow}a\right]\right]$
(see \cite[Definition 8]{Celani QUasi modal}). 

We note that if $h$ satisfies condition (CO), then $0=h(0)=h(a\wedge\neg a)=h(a)\wedge h(\neg a)$
iff $h(a)\leq\neg h(\neg a)=g(a)$, for each $a\in A_{1}$. Thus,
if $h$ satisfies (CO) and (DH), then $h$ satisfies (MO). 

Now we are going to characterize the condition (MO).
\begin{lem}
\label{lem:morfismo subordination}Let $\left<A,\prec_{1}\right>,$$\left<A_{2},\prec_{2}\right>\in\mathrm{SubS4}$.
Let $h:A_{1}\rightarrow A_{2}$ be a meet-homomorphism. Then the following
conditions are equivalent:
\begin{enumerate}
\item $h$ satisfies the condition $\mathrm{(MO)}$, 
\item \label{condition subordination relation}$h_{R}(\mathrm{int}_{S_{1}}(U))\subseteq\mathrm{int}_{S_{2}}(h_{R}(U))$,
for all $U\in\mathrm{Clop}(X_{1})$,
\item $h_{R}\left[\downarrow_{S_{1}}(U)\right]\subseteq\downarrow_{S_{2}}(h_{R}(U))$,
for all $U\in\mathrm{Clop}(X_{1})$.
\end{enumerate}
\end{lem}

\begin{proof}
$\left(1\right)\Rightarrow\left(2\right)$ Let $U=\varphi_{1}(a)$.
First, let us prove that $h_{R}(\beta(\downarrow_{1}a))=\beta(h\left[\downarrow_{1}a\right])$.
Let $x\in X_{2}$ such that $R(x)\subseteq\beta(\downarrow_{1}a))$.
If $x\notin\beta(h\left[\downarrow_{1}a\right])$, then $\downarrow_{1}a\cap h^{-1}(x)\neq\emptyset$.
Then there exists $y\in X_{1}$ such that $h^{-1}(x)\subseteq y$
and $\downarrow_{1}a\cap y=\emptyset$, i.e., $y\in R(x)\subseteq\beta(\downarrow_{1}a)),$i.e.,
$\downarrow_{1}a\cap y\neq\emptyset$, which is a contradiction. The
other inclusion is proved similarly. 

Then 
\begin{eqnarray*}
h_{R}(\mathrm{int}_{S_{1}}(\varphi_{1}(a))) & = & h_{R}(\beta(\downarrow_{1}a))\\
 & = & \beta(h\left[\downarrow_{1}a\right])\\
 & \subseteq & \beta(\downarrow_{2}h(a))\\
 & = & \mathrm{int}_{S_{2}}(h_{R}(\varphi_{1}(a))).
\end{eqnarray*}

$\left(2\right)\Rightarrow\left(3\right)$ Let $W\in h_{R}\left[\downarrow_{S_{1}}(U)\right]$.
Then there exists $V\prec_{S_{1}}U$ such that $W=h_{R}(V)$. From
$V\subseteq\mathrm{int}_{S_{1}}(U)$ we have $W=h_{R}(V)\subseteq h_{R}(\mathrm{int}_{S_{1}}(U))\subseteq\mathrm{int}_{S_{2}}(h_{R}(U))$.
Then $W\prec_{S_{2}}h_{R}(U)$, i.e., $W\in\downarrow_{S_{2}}(h_{R}(U))$. 

$\left(3\right)\Rightarrow\left(2\right)$ Let $x\in h_{R}(\mathrm{int}_{S_{1}}(U))$.
So, $R(x)\subseteq\mathrm{int}_{S_{1}}(U)$. As $\mathrm{int}_{S_{1}}(U)$
is an open of $\tau$, for each $y\in R(x)$ there exists $W_{y}\in\mathrm{Clop}(X_{1})$
such that $y\in W_{y}$ and $W_{y}\subseteq\mathrm{int}_{S_{1}}(U)$.
So, $R(x)\subseteq\bigcup\left\{ W_{y}\in\mathrm{Clop}(X_{1}):y\in W_{y}\subseteq\mathrm{int}_{S_{1}}(U)\right\} $.
Since $R(x)$ is closed in $\left<X_{2},\tau\right>$, it is compact.
So, there exists $W_{y_{1}},\ldots,W_{y_{n}}\in\mathrm{Clop}(X_{1})$
such that $R(x)\subseteq W=W_{y_{1}}\cup\ldots\cup W_{y_{n}}\subseteq\mathrm{int}_{S_{1}}(U)$.
Then, $x\in h_{R}(W)$ and $W\subseteq\mathrm{int}_{S_{1}}(U)$. As
$h_{R}(W)\in h_{R}\left[\downarrow_{S_{1}}(U)\right]\subseteq\downarrow_{S_{2}}(h_{R}(U))$,
we have $h_{R}(W)\subseteq\mathrm{int}_{S_{2}}(h_{R}(U))$, and thus
$x\in\mathrm{int}_{S_{2}}(h_{R}(U))$.
\end{proof}

\begin{lem}
\label{lem:round-homo}Let $\left<A,\prec_{1}\right>,$$\left<A_{2},\prec_{2}\right>\in\mathrm{SubS4}$.
Let $h:A_{1}\rightarrow A_{2}$ be a meet-homomorphism satisfying
$\mathrm{(MO)}$. Then the following conditions are equivalent:
\begin{enumerate}
\item For every round filter $F$ of $A_{2}$, $h^{-1}\left[F\right]$ is
a round filter of $A_{1}$.
\item $h$ satisfies the condition $\mathrm{(QH)}$.
\item \label{condition strong sub relation}$\mathrm{int}_{S_{2}}(h_{R}(U))\subseteq h_{R}(\mathrm{int}_{S_{1}}(U))$,
for all $U\in\mathrm{Clop}(X_{1})$. 
\end{enumerate}
\end{lem}

\begin{proof}
$\left(1\right)\Rightarrow\left(2\right)$ We note that to prove (2)
is equivalent to prove that ${\downarrow}h(a)\subseteq\left(h({\downarrow}a)\right]$,
for all $a\in A_{2}$. Suppose that there exist $a\in A_{2}$ such
that $\downarrow h(a)\nsubseteq\left(h(\downarrow a)\right].$ Then
there exists an ultrafilter $x$ of $A_{2}$ such that ${\downarrow}h(a)\cap x\neq\emptyset$
and $h({\downarrow}a)\cap x=\emptyset$. So, $h(a)\in\uparrow(x)=F$,
and thus $a\in h^{-1}(\uparrow x)=h^{-1}(F)$. As $F$ is a round
filter, we have that $h^{-1}(F)$ is a round filter. So, ${\downarrow}a\cap h^{-1}(F)\neq\emptyset$,
i.e, there exist $c\prec a$ and $h(c)\in F\subseteq x$. Since, $c\in{\downarrow}a$
implies $h(c)\in h({\downarrow}a)$, we deduce that $h(c)\in h({\downarrow}a)\cap x$,
which is a contradiction. Therefore, ${\downarrow}h(a)\nsubseteq\left(h({\downarrow}a)\right]$,
for all $a\in A$. 

$\left(2\right)\Rightarrow\left(3\right)$ Let $U=\varphi_{1}(a)$.
Let $x\in\mathrm{int}_{S_{2}}(h_{R}(U))=\beta_{2}({\downarrow}h(a))$.
Let $y\in X_{2}$ such that $h^{-1}(x)\subseteq y$. We prove that
${\downarrow}a\cap y\neq\emptyset$. As $x\in\beta_{2}({\downarrow}h(a))$,
we have ${\downarrow}(h(a))\cap x\neq\emptyset$. Then there exists
$b\in A_{2}$ such that $b\in x$ and $b\prec h(a)$. Then there exists
$c\in A_{1}$ such that $b\leq h(c)$ and $c\prec a$. So, $h(c)\in x$,
i.e., $c\in h^{-1}(x)\subseteq y$. So, ${\downarrow}a\cap y\neq\emptyset$,
and thus $y\in\beta({\downarrow}a)$. Therefore, $h_{R}(\beta({\downarrow}a))=h_{R}(\mathrm{int}_{S_{1}}(\varphi(a)))$.

$\left(3\right)\Rightarrow\left(1\right)$ Let $F$ be a round filter
of $A_{2}$. As $F$ is a lattice filter, $h^{-1}\left[F\right]$
is a lattice filter. Let $b\in h^{-1}\left[F\right]$. Then $h(b)\in F$,
and as $F$ is a round filter, there exists $a\in A_{2}$ such that
$a\prec h(b)$ and $a\in F$. Then, $\varphi(a)\prec_{S_{2}}\varphi(h(b))=h_{R}(\varphi(b))$,
i.e, 
\[
\varphi(a)\subseteq\mathrm{int}_{S_{2}}h_{R}(\varphi(b))\subseteq h_{R}(\mathrm{int}_{S_{1}}(\varphi(b)).
\]
We note that $\mathrm{int}_{S_{1}}(\varphi(b)=\bigcup\left\{ W\in\Clop(X_{1}):W\subseteq\mathrm{int}_{S_{1}}\varphi(b)\right\} $
because $\mathrm{int}_{S_{1}}(\varphi(b)$ is an open in the Stone
space $\left<\Ul(A_{1}),\tau\right>$. So, 
\[
h_{R}(\mathrm{int}_{S_{1}}(\varphi(b))=\bigcup\left\{ h_{R}(W):W\subseteq\mathrm{int}_{S_{1}}\varphi(b)\right\} .
\]
 Therefore, $\varphi(a)\subseteq\bigcup\left\{ h_{R}(W):W\subseteq\mathrm{int}_{S_{1}}\varphi(b)\right\} $,
and as $\varphi(a)$ is compact, there exist $W_{1},\ldots,W_{n}\subseteq\mathrm{int}_{S_{1}}\varphi(b)$
such that 
\[
\varphi(a)\subseteq h_{R}(W_{1})\cup\ldots\cup h_{R}(W_{n})\subseteq h_{R}(W_{1}\cup\ldots\cup W_{n}).
\]
 Let $W_{1}\cup\ldots\cup W_{n}=\varphi(c)$. Then, $\varphi(a)\subseteq h_{R}(\varphi(c))=\varphi(h(c))$,
i.e., $a\leq h(c)$ and $c\prec b$. Thus we have proved that of there
exists $c\in A$ such that $a\leq h(c)$ and $c\prec b$. Since $a\in F$,
$c\in h^{-1}(F)$, and as, $c\prec b$, we have that $h^{-1}\left[F\right]$
is a round filter. 
\end{proof}

By Lemma \ref{lem:morfismo subordination} and \ref{lem:round-homo}
we have the following characterization of meet-homomorphisms satisfying
the conditions (MO) and (QH). 
\begin{cor}
\label{cor: morfi (MO)+(QH)}Let $\left<A,\prec_{1}\right>,$$\left<A_{2},\prec_{2}\right>\in\mathrm{SubS4}$.
Let $h:A_{1}\rightarrow A_{2}$ be a meet-homomorphism. Then the following
conditions are equivalent:
\begin{enumerate}
\item $h$ satisfies the conditions $\mathrm{(MO)}$ and $\mathrm{(QH)}$.
\item $h_{R}(\mathrm{int}_{S_{1}}(U))=\mathrm{int}_{S_{2}}(h_{R}(U))$,
for all $U\in\mathrm{Clop}(X_{1})$.
\end{enumerate}
\end{cor}

Now we will characterize the meet-homomorphisms fulfilling the conditions
(CO) and (DH). 
\begin{lem}
\label{charact hemi1}Let $\left<A,\prec_{1}\right>,$$\left<A_{2},\prec_{2}\right>\in\mathrm{SubS4}$.
Let $h:A_{1}\rightarrow A_{2}$ be a meet-homomorphism such that $h(0)=0$.
Then the following conditions are equivalent: 
\begin{enumerate}
\item $h$ satisfies the condition $\mathrm{(DH})$. 
\item \label{condition weak de vries}$g_{R}(\mathrm{int}_{S_{1}}(U))\subseteq\mathrm{int}_{S_{2}}(h_{R}(U))$,
for all $U\in\mathrm{Clop}(X_{1})$,
\item $g_{R}\left[{\downarrow}_{S_{1}}(U)\right]\subseteq{\downarrow}_{S_{2}}(h_{R}(U))$,
for all $U\in\mathrm{Clop}(X_{1})$.
\end{enumerate}
\end{lem}

\begin{proof}
$\left(1\right)\Rightarrow\left(2\right)$ Let $U\in\mathrm{Clop}(X_{1})$.
Then there exists $b\in A_{1}$ such that $U=\varphi(b)$. We can
suppose that $\mathrm{int}_{S_{1}}(\varphi(b))\neq\emptyset$. Then
there exists $a\in A_{1}$ such that $\varphi(a)\subseteq\mathrm{int}_{S_{1}}(\varphi(b))$,
i.e. $\varphi(a)\prec_{S_{1}}\varphi(b)$, and so $a\prec b$. By
condition (DH) we get $g(a)\prec h(b)$, and consequently $\varphi(g(a))=g_{R}(\varphi(a))\prec_{S_{2}}\varphi(h(b))$.
So, $g_{R}(\varphi(a))\subseteq\mathrm{int}_{S_{2}}(\varphi(h(b)))$,
and as $\mathrm{int_{S_{1}}}(\varphi(a))\subseteq\varphi(a)$, we
have $g_{R}(\mathrm{(int}_{S_{1}}(\varphi(a)))\subseteq\mathrm{int}_{S_{2}}(\varphi(h(b)))$.

$\left(2\right)\Rightarrow\left(1\right)$ Let $a\prec_{1}b$. Then
$\varphi(a)\subseteq\mathrm{int}_{S_{1}}(\varphi(b))$, and as $g_{R}$
is monotonic we have
\[
g_{R}(\varphi(a))=\varphi(g(a))\subseteq g_{R}(\mathrm{int}_{S_{1}}(\varphi(b))\subseteq\mathrm{int}_{S_{2}}(h_{R}(\varphi(b)))=\mathrm{int}_{S_{2}}(\varphi(h(b)))).
\]
Thus, $g(a)\prec_{2}h(b)$.

$\left(2\right)\Rightarrow\left(3\right)$ Let $V\in{\downarrow}_{S_{1}}(U)$,
i.e., $V\subseteq\mathrm{int}_{S_{1}}(U)$. As $g_{R}$ is monotonic,
$g_{R}(V)\subseteq g_{R}(\mathrm{int}_{S_{1}}(U))\subseteq\mathrm{int}_{S_{2}}(h_{R}(U))$.
Then, $g_{R}(V)\in{\downarrow}_{S_{2}}(h_{R}(U))$. Thus, $g_{R}\left[{\downarrow}_{S_{1}}(U)\right]\subseteq{\downarrow}_{S_{2}}(h_{R}(U))$. 

$\left(3\right)\Rightarrow\left(2\right)$ Let $x\in g_{R}(\mathrm{int}_{S_{1}}(U))$.
So, $R(x)\cap\mathrm{int}_{S_{1}}(U)$. Then there exists $y\in R(x)$
and $y\in\mathrm{int}_{S_{1}}(U)$. So $\mathrm{int}_{S_{1}}(U)\in\tau$,
there exists $W\in\mathrm{Clop}(X_{1})$ such that $y\in W\subseteq\mathrm{int}_{S_{1}}(U)$.
Then $W\in{\downarrow}_{S_{1}}(U)$, and so $g_{R}(W)\in g_{R}\left[{\downarrow}_{S_{1}}(U)\right]\subseteq{\downarrow}_{S_{2}}(h_{R}(U))$.
Thus, $g_{R}(W)\subseteq\mathrm{int}_{S_{2}}(h_{R}(U))$, and as $x\in g_{R}(W)$,
we get that $x\in\mathrm{int}_{S_{2}}(h_{R}(U))$.
\end{proof}

By Lemma \ref{lem:round-homo} and Lemma \ref{charact hemi1} we have
the following characterization of meet-homomorphisms satisfying the
conditions$\mathrm{(CO)}$, (QH) and $\mathrm{(DH)}$.
\begin{cor}
Let $\left<A,\prec_{1}\right>,$$\left<A_{2},\prec_{2}\right>\in\mathrm{SubS4}$.
Let $h:A_{1}\rightarrow A_{2}$ be a meet-homomorphism. Then the following
conditions are equivalent
\begin{enumerate}
\item $h$ satisfies the conditions $\mathrm{(CO)}$, $\mathrm{(QH)}$ and
$\mathrm{(DH)}$.
\item $R(x)\neq\emptyset$, for all $x\in X(A_{2})$, $g_{R}(\mathrm{int}_{S_{1}}(U))\subseteq\mathrm{int}_{S_{2}}(h_{R}(U))$,
and $\mathrm{int}_{S_{2}}(h_{R}(U))\subseteq h_{R}(\mathrm{int}_{S_{1}}(U))$,
for all $U\in\mathrm{Clop}(X_{1})$.
\end{enumerate}
\end{cor}

It is easy to verify that the composition of meet-homomorphisms satisfying
any of the conditions (CO), (MO), (QH) or (DH) is a meet-homomorphism
satisfying the same conditions. Dually, it can be verified that the
same occurs with the Boolean relations associated with the meet-homomorphisms.

From the preceding results, one can define several categories of S4-subordination
algebras whose morphisms are meet-homomorphisms satisfying one of
the following sets of conditions: $\mathrm{(MO)}$, $\mathrm{(MO)}+\mathrm{(QH)}$,
$\mathrm{(DH)}$, or $\mathrm{(CO)}+\mathrm{(QH)+(DH)}$. As an illustration,
the category of S4-subordination algebras with meet-homomorphisms
satisfying $\mathrm{(MO)+(QH)}$ is dually equivalent to the category
of S4-subordination spaces whose morphisms are Boolean relations satisfying
condition (2) of Corollary \ref{cor: morfi (MO)+(QH)}. It is left
to the reader to complete the details of the different dualities that
can be proven. 

\medskip{}

We now present a characterization of meet-homomorphisms satisfying
condition (DH) in terms of maximal round filters, also called ends
in \cite{de Vries tesis}. The following characterization of ends
is useful
\begin{lem}
\cite{de Vries tesis} Let $\left<A,\mathrm{\prec}\right>\in\mathrm{SubS5}$.
A round filter $F$ is an end iff for $a,b\in A$, if $a\prec b$,
then $\neg a\in F$ or $b\in F$. 
\end{lem}

\begin{lem}
\label{characterizacio  end}Let $\left<A,\prec_{1}\right>,\left<A_{2},\prec_{2}\right>\in\mathrm{SubS5}$.
Let $h:A_{1}\rightarrow A_{2}$ be a meet-homomorphism. Then the following
conditions are equivalent:
\begin{enumerate}
\item If $F\in\mathrm{End}(A_{2})$, and $a\prec b$, then $\neg a\in h^{-1}(F)$
or $b\in h^{-1}(F)$, for $a,b\in A_{1}$. 
\item If $a\prec b$, then $g(a)\prec h(b),$ for all $a,b\in A_{1}$.
\end{enumerate}
\end{lem}

\begin{proof}
$\left(1\right)\Rightarrow\left(2\right)$ Suppose that there are
elements $a,b\in A$ such that $a\prec b$ but $g(a)\nprec h(b)$,
i.e., $g(a)\notin{\downarrow}h(b)$. Then there exists an ultrafilter
$P$ of $A_{2}$ such that $g(a)\in P$ and $P\cap{\downarrow}h(b)=\emptyset$.
As $A_{2}\in\mathsf{SubS5},$ $F={\uparrow}P$ is an end. So, $b\notin h^{-1}(F)$.
Since, $a\prec b$, we get $\neg a\in h^{-1}(F)$ or $b\in h^{-1}(F)$.
If $\neg a\in h^{-1}(F)$, we have $h(\neg a)\in F={\uparrow}P\subseteq P$,
but as $g(a)=\neg h(\neg a)\in P$, we get $h(\neg a)\wedge\neg h(\neg a)=0\in P$,
which is impossible. So, $b\in h^{-1}(F)$, i.e., $h(b)\in F={\uparrow}P$,
which is also a contradiction. Therefore, $g(a)\prec h(a)$. 

$\left(2\right)\Rightarrow\left(1\right)$ Suppose that $a\prec b$.
Let $F\in\mathrm{End}(A_{2})$. As $g(a)=\neg h(\neg a)\prec h(b)$
y $F$ is an end, $\neg\neg h(\neg a)=h(\neg a)\in F$ or $h(b)\in F$,
i.e, $\neg a\in h^{-1}(F)$ or $b\in h^{-1}(F)$. 
\end{proof}

Finally, we give an algebraic characterization of the meet-homomorphisms
satisfying the conditions (CO), (QH) and (DH) between two S5-subordination
algebras.
\begin{thm}
\label{cor:fundamental}Let $A_{1}$ and $A_{2}$ be two S5-subordination
algebras. Then a map $h:A_{1}\rightarrow A_{2}$ is a meet-homomorphism
satisfying the conditions $\mathrm{(CO)}$, $\mathrm{(QH)}$ and $\mathrm{(DH)}$
if and only if
\begin{enumerate}
\item $h^{-1}(F)$ is a proper filter of $A_{1}$ for each proper filter
$F$ of $A_{2}$.
\item $h^{-1}(F)$ is a round filter of $A_{1}$ for each round filter $F$
of $A_{2}$.
\item $h^{-1}(P)$ is an end of $A_{1}$ for each end $P$ of $A_{2}$.
\end{enumerate}
\end{thm}

\begin{proof}
$\Rightarrow)$ As $h$ is a meet-homomorphism, $h^{-1}(F)\in\mathrm{Fi}(A_{1})$
for every $F\in\mathrm{Fi}(A_{2})$. By the condition $h(0)=0$ we
have that $h^{-1}(F)$ is proper whenever $F$ is proper. The condition
(2) follow by Lemma \ref{lem:round-homo}. We prove (3). Let $P\in\mathrm{End}(A_{2})$.
Since $h$ is a meet-homomorphism satisfying the condition (DH), we
get that $h$ satisfies (MO). Then, by \ref{lem:round-homo}, $h$$^{-1}\left[P\right]$
is a round filter. By (1) $h^{-1}\left[P\right]$ is proper. By Lemma
\ref{characterizacio  end} we have that $h^{-1}\left[P\right]$ is
an end.

$\Leftarrow)$ From (1) we have that $h$ is a meet-homomorphism satisfying
the condition (CO). By Lemma \ref{lem:round-homo} we have that $h$
satisfies the condition (QH). Finally, the condition (DH) follows
by Lemma \ref{characterizacio  end}.
\end{proof}

\section*{Acknowledgements }

We thank the referee for their valuable comments and suggestions.

\section*{Declarations}

\subsection*{Founding}

This work was supported by Consejo Nacional de Investigaciones Científicas
Técnicas (PIP 11220200100912CO, CONICET, Argentina), Universidad Nacional
del Centro de la Provincia de Buenos Aires, Facultad de Ciencias Exactas
and Agencia Nacional de Promoción Científica y Tecnológica (PICT2019-2019-00882,
ANPCyT-Argentina). This project has also received funding from MOSAIC
Project 101007627 (European Union’s Horizon 2020 research and innovation
programme under the Marie Sklodowska-Curie).

\subsection*{Conflict of interest}

The authors declare no conflict of interest.

\subsection*{Ethics approval and consent to participate}

Not applicable.

\subsection*{Consent for publication}

Not applicable.

\subsection*{Data availability}

Not applicable.

\subsection*{Materials availability}

Not applicable.

\subsection*{Code availability}

Not applicable.

\subsection*{Author contribution}

Not applicable.

\end{document}